\documentclass{amsart}
\usepackage{amsmath,amssymb,mathrsfs,graphicx}
\usepackage[utf8x]{inputenc}
\usepackage{amsmath}
\usepackage{amsfonts}
\usepackage{latexsym}
\usepackage{amsthm}
\usepackage{ulem}
\usepackage{amssymb,amscd}
\usepackage{xargs}
\usepackage{tikz}

\usepackage{graphicx}
\usepackage{color}
\usepackage{enumitem}
\usepackage{todonotes}
\newtheorem{thm}{Theorem}[section]

\newtheorem{lem}[thm]{Lemma}
\newtheorem{defn}[thm]{Definition}
\newtheorem{prop}[thm]{Proposition}
\newtheorem{cor}[thm]{Corollary}

\newtheorem{rmk}[thm]{Remark}
\newcommand{\be}{\begin{eqnarray}}
\newcommand{\ee}{\end{eqnarray}}
\newcommand{\beal}{\begin{aligned}}
\newcommand{\enal}{\end{aligned}}

\newcommand{\eps}{\varepsilon}

\newcommand{\al}{\alpha}
\newcommand{\gm}{\gamma}
\newcommand{\ga}{\gamma}

\newcommand{\lb}{\lambda}

\newcommand{\C}{\mathbb{C}}
\newcommand{\N}{\mathbb{N}}

\newcommand{\T}{\mathbb{T}}
\newcommand{\R}{\mathbb{R}}
\newcommand{\Q}{\mathbb{Q}}

\newcommand{\bY}{\mathbb{Y}}

\newcommand{\ZZ}{\mathbb{Z}}
\newcommand{\RR}{\mathbb{R}}
\newcommand{\om}{\omega}

\newcommand{\rr}{\rho}
\newcommand{\de}{\delta}

\newcommand{\cU}{\mathcal{U}}
\newcommand{\cV}{\mathcal{V}}

\newcommand{\cD}{\mathcal{D}}

\newcommand{\cS}{\mathcal{S}}

\newcommand{\cB}{\mathcal{B}}

\newcommand{\cC}{\mathcal{C}}

\newcommand{\cG}{\mathcal{G}}
\newcommand{\cT}{\mathcal{T}}

\newcommand{\wh}{\widehat }
\newcommand{\wt}{\widetilde }

\newcommand{\pol}{\mathrm{pol}}
\newcommand{\lin}{\mathrm{lin}}

\newcommand{\col}{\mathrm{col}}

\title[Asymptotic density of collision orbits in the RCPTBP]{Asymptotic density of collision orbits in the Restricted Circular Planar 3 Body Problem}
\author{Marcel Guardia\dag}
\address{\dag\ Departament de Matem\`atiques, Universitat Polit\`ecnica de Catalunya}
\email{marcel.guardia@upc.edu}
\author{Vadim Kaloshin\ddag}
\address{\ddag\  University of Maryland, College Park, MD, USA}
\email{vadim.kaloshin@gmail.com}
\author{Jianlu Zhang*}
\address{*\ Institute of Theoretical Studies, ETH Z\"urich, CH-8092 Z\"urich, Switzerland}
\email{jellychung1987@gmail.com}
\thanks{This project has received funding from the European Research Council (ERC) under the European Union’s Horizon 2020 research and innovation programme (grant agreement No 757802). M. G. has been also partially supported by the Spanish MINECO-FEDER Grant MTM2015-65715-P. 
V.K. acknowledges partial support of the NSF grant DMS-1402164. The three 
authors acknowledges the hospitality of the ETH Institute for Theoretical 
Studies and the support of Dr. Max
R\"ossler, the Walter Haefner Foundation and the ETH Zurich Foundation. }
\date{}
\begin{document}
\maketitle
\begin{abstract}
For the Restricted Circular Planar  3 Body Problem, we show that there exists 
an open set $\cU$ in phase space independent of fixed measure,  where the set 
of initial points which lead to collision is 
$O(\mu^\frac{1}{20})$ dense as $\mu\rightarrow 0$.
\end{abstract}
\section{Introduction }

\medskip 
Understanding solutions of the Newtonian 3 body problem 
is a long standing classical problem. There is not much of 
a hope to give a precise answer given an initial condition. 
However, one hopes to give a qualitative classification. For example, 
divide solutions into several classes according to qualitative asymptotic
behavior and describe geometry and measure theoretic 
properties of each set. The first such an attempt probably 
goes back to Chazy  \cite{Cha}.  

\subsection{Chazy's classification and Kolmogorov's conjecture}
Ignore solutions not defined for all times,  then one possible direction is 
to study qualitative behavior of bodies as time tends to 
infinity either in the future or in the past. 
In 1922 Chazy \cite{Cha} gave a classification of all possible types 
of asymptotic motions (see also \cite{AKN}). Denote $r_k$ the vector from $q_i$ 
to 
$q_j$ with $i \ne k,\ j \ne k, i < j$. 

\begin{thm}[Chazy, 1922]
Every solution of the 3 body problem defined for all times belongs to one of 
the following seven classes: 
\begin{itemize}
\item $\mathcal H^+$ (hyperbolic): $|r_k| \to \infty,\ |\dot r_k| \to c_k\ne 0$ as $t \to +\infty$; 
\item $ \mathcal{HP}^+_k$ (hyperbolic-parabolic): $|r_k| \to \infty,\ |\dot 
r_k| \to 0, \ |\dot r_i| \to c_i > 0\ (i \ne k)$; 
\item $\mathcal{HE}^+_k$ (hyperbolic-elliptic): $|r_k| \to \infty,\ |\dot r_i| \to c_i > 0 (i\ne k),\ 
\sup_{t\ge 0} |r_k| < \infty$; 
\item $\mathcal{PE}^+_k$ (hyperbolic-elliptic): $|r_k| \to \infty,\ |\dot r_i| \to 0\ (i \ne k),\ \sup_{t\ge 0} |r_k| < \infty;$
\item $\mathcal P_+$ (parabolic) $|r_k| \to \infty,\ |\dot r_k| \to 0$; 
\item $\mathcal B^+$ (bounded): $\sup_{t\ge 0} |r_k| < \infty$; 
\item $\mathcal{OS}^+$ (oscillatory): $\lim \sup_{t\to \infty} \max_k |r_k| = \infty,\ \lim \inf_{t\to \infty} \max_k |r_k| < \infty$. 
\end{itemize}
\end{thm}

Examples of the first six types were known to Chazy. The existence of oscillatory motions was proved by
Sitnikov \cite{Si} in 1959. The next natural question is to evaluate the measure of each of these sets. 
It turns out that the answer is known for all sets except one, see the table below. The remaining set 
is the set of oscillatory motions. Proving or disproving that this set has measure zero is the central 
problem in qualitative analysis of the 3 body problem.

\vskip 0.25in

\begin{center}
\begin{tabular}{|c|c|c|}
\hline
\multicolumn{3}{|c|}{Positive energy $H>0$} \\ \hline
 & $H^+$ & $HE_i^+$ \\ \hline
 & & \\
 & Lagrange, 1772  & PARTIAL CAPTURE \\
 & (isolated examples); & Measure $>0$ \\
$H^-$ & Chazy, 1922  & Shmidt (numerical examples), 1947; \\
 & Measure $>0$ & Sitnikov (qualitative methods), 1953 \\ \hline
 & & \\
 & COMPLETE DISPERSAL & $i=j$ \quad Measure $>0$ \\
 & Measure $>0$ & Birkhoff, 1927 \\ \cline{3-3}
  $HE_j^-$  & & $i \ne j$ \quad EXCHANGE, Measure $>0$ \\
  & & Bekker (numerical examples), 1920; \\
  & & Alexeev (qualitative methods), 1956 \\ \hline
\end{tabular}

\vskip 0.5in

\begin{tabular}{|c|c|c|c|}
\hline
\multicolumn{4}{|c|}{Negative energy $H<0$} \\ \hline
    & $HE_i^+$ & $B^+$ & $OS^+$ \\ \hline
 & & & \\
 & $i=j$ \quad  Measure $>0$ & COMPLETE &  \\
 & Birkhoff, 1927 & CAPTURE & \\ \cline{2-2}
 & Exchange & Measure $=0$ & Measure $=0$ \\
 & $i\ne j$ \quad Measure $>0$ & Chazy,1929 \& Merman,1954; &
 Chazy,1929 \& Merman,1954; \\
 $HE_j^-$ & Bekker, 1920 & Littlewood, 1952; & Alexeev, 1968 \\
 & (numerical examples); & Alexeev, 1968; & $\ne \emptyset$ \\
 &  Alexeev, 1956; & $\ne \emptyset$ & \\
 & (qualitative methods) & & \\  \hline
 & & & \\
 & PARTIAL & Euler, 1772; & Littlewood, 1952 \\
 & DISPERSAL & Lagrange, 1772 & Measure $= 0$ \\
  $B^-$ & $\ne \emptyset$ & Poincare, 1892 & Alexeev, 1968 \\
 & Measure $= 0$ & (isolated examples); & $\ne \emptyset$ \\
 &  & Arnold, 1963 & \\ \hline & & & \\
 & $\ne \emptyset$ & $\ne \emptyset$ & Sitnikov, 1959, \\
     $OS^-$  & Measure $= 0$ & Measure $= 0$ & $\ne \emptyset$ \\
      &  & & Measure $= ?$ \\ \hline
\end{tabular}
\end{center}

\vskip 0.25in

Thus, the remaining major open problem is the following 
\medskip 

\begin{center} {\it  Conjecture (\,{\bf Kolmogorov})}
{\it The set of oscillatory motions has zero Lebesgue measure. }
\footnote{In \cite{A} Alexeev  attributes the conjecture that 
the set of oscillatory motions has measure zero to Kolmogorov. 
In \cite{A2} Kolmogorov is not mentioned.}
\end{center}

\subsection{The oldest open question in dynamics
and non-wandering orbits}

Now we give a different look at the classification of 
qualitative behavior of solutions. 
In the 1998 International Congress of Mathematicians, 
Herman \cite{Her} ended his beautiful survey of open 
problems with the following question, which he called
{\it``the oldest open question in dynamical systems''}. 
Let us recall the definition of a non-wandering point.
\begin{defn}\label{defi:wandering}
Consider a a dynamical system $\{\phi_t\}_{t\in\RR}$ defined on a topological space $X$. 
Then, a point $x\in X$  is called {\sf wandering}, if there exists a neighborhood $\cV$ of it and $T>0$, 
such that $\phi(t, \cV)\cap\cV=\emptyset$ for all $t>T$.

Conversely, $x\in X$ is called {\sf non wandering}, if for any neighborhood 
$\cV$ of $z$ and 
any $T>0$, there  exists $t>T$ such that $\phi(t, 
\cV)\cap\cV\neq\emptyset$.
\end{defn}

Consider the  $N$-body problem in space with  $N\ge 3$. Assume that,
\begin{itemize}
\item The center of mass is fixed at 0. 
\item On the energy surface we $C^\infty$-reparametrize the flow 
by a $C^\infty$ function $\psi_E$ (after reduction of the center of 
mass) such that the flow is complete: we replace $H$ by 
$\psi_E(H_E) = H_E$ so that the new flow takes an infinite time to 
go to collisions ($\psi_E$ is a $C^\infty$ function). 
\end{itemize}
Following Birkhoff \cite{Bi} (who only considers the case 
$N = 3$ and nonzero angular momentum) (see also 
Kolmogorov \cite{K}), Herman asks the following question:
\medskip 

{\bf Question 1} {\it Is for every $E$ the nonwandering set of the Hamiltonian 
flow of $H_E$ on $H^{-1}_E (0)$ nowhere dense in 
$H^{-1}_E (0)$?}

In particular, this would imply that 
the bounded orbits are nowhere dense and no topological 
stability occurs.  \medskip 
 
%

It follows from the identity of Jacobi-Lagrange that when 
$E \ge 0$, every point such that its orbit is defined for all times, 
is wandering. The only thing known is that, even when $E < 0$, 
wandering sets do exist (Birkhoff and Chazy, see Alexeev 
\cite{A} for references). 
\smallskip

The fact that the bounded orbits have positive Lebesgue-measure 
when the masses belong to a non empty open set, is a remarkable 
result announced by Arnold \cite{Ar} (Arnold gave only a proof 
for the planar 3 body problem;  see also
\cite{Rob1, Rob2, CP,Fe2}). 
In some respect Arnold's claim proves 
that Lagrange and Laplace, who believed on the stability of the Solar system, are correct in the sense of 
measure theory. On the contrary, in the sense of topology, the above 
question, in some respect, could show Newton, who believed the Solar system to be unstable, to be correct.

\subsection{Collisions are frequent, are they? }

The above discussion relies on solutions 
being well defined for {\it all time.} It leads to 
the analysis of the set of solutions with a collision.
Saari \cite{Sa1,Sa2} (see also \cite{Knauf1, Knauf2}) proved that this set has 
zero measure. 
However, they might form a topologically ``rich'' set. 
Here is a question which is proposed by Alekseev \cite{A}
and might be traced back to Siegel, Sec. 8, P. 49 in \cite{S}.\medskip

{\bf Question 2} {\it Is there an open subset $\cU$ of the phase space 
such that for a dense subset of initial conditions the associated 
trajectories go to a collision? }\medskip

The geometric structure of the collision manifolds locally was given 
by Siegel in \cite{S}, by applying the Sundmann regularization of double collisions. 
But the above question is still open. In the current article we consider a 
special case: 
the restricted planar circular 3 body problem and give a partial answer.\medskip 

Marco and Niederman \cite{MN},  Bolotin and McKay \cite{BM,BM1} and Bolotin \cite{B, B1, B2} studied collision and near collision solutions.
Chenciner--Libre \cite{CL} and Fejoz \cite{F} constructed so-called 
punctured tori, i.e. tori with quasiperiodic motions passing through 
a double collision (see also  \cite{Z}). In this paper we only deal with double collisions. 
Triple collisions have also been thoroughly studied (see \cite{Moe2, Moe3, 
Moe1} and 
references therein). 

\subsection
{\bf Restricted Circular Planar 3 Body Problem (RCP3BP)} 
Consider  two massive bodies (the primaries), which we
call {\it the Sun} and {\it Jupiter}, moving under the influence of the mutual Newton gravitational force. Assume they perform circular motion. 
We can normalize the mass of Jupiter by $\mu$ and the
Sun by $1-\mu$ and fix  the center of mass  at zero. The {\it restricted planar circular 3 body problem (RPC3BP)} models the dynamics of a third body, which we call
 {\it the Asteroid}, that  has mass zero and moves by the influence of
the gravity of the primaries. In rotating coordinates, the dynamics of the
Asteroid is given by the Hamiltonian 
\be\label{r3b-e}
H_\mu(x,y)=\frac{|y|^2}{2}- x^t Jy-\frac{\mu}{|x-(1-\mu,0)|}-\frac{1-\mu}{|x-(-\mu,0)|},
\ee
where $x\in\R^2$ is the position, $y\in\R^2$ is the conjugate momentum and 
\[
J=\begin{pmatrix}
 0     & 1   \\
   -1   &  0
\end{pmatrix}
\]
is the standard symplectic matrix. The positions of the
primaries are always fixed at $(-\mu,0)$ (the Sun) and $(1-\mu,0)$ (Jupiter)
respectively. In addition, the system is conservative and  $J=-2H_\mu(x,y)$ is
called the  {\it Jacobi Constant}.

An orbit $\gamma(t)=(x(t),y(t))$ of
(\ref{r3b-e}) is called a {\it collision orbit}, if in finite time $T$ we have
either  $x(T)=(1-\mu,0)$ or $x(T)=(-\mu,0)$. Then, Siegel question can be
rephrased as whether there exists an open set $\mathcal U$ in phase space
independent of $\mu$ where the collision orbits are dense. The main result of
this paper is the following.

\begin{thm}[First main Result]\label{main}
There  exists an open set $\cU$ independent of $\mu>0$ where the 
collision orbits of the Hamiltonian $H_\mu$  in (\ref{r3b-e}) are
$O(\mu^{\frac{1}{20}})$ dense as $\mu$ tends to zero.
 \end{thm}

To  explain heuristically this result, consider first the case $\mu=0$.
Since for $\mu=0$ the system is integrable, any energy surface $\{H_0=h\}$ 
is foliated by invariant $2$-dimensional tori. They correspond to circular orbits of 
Jupiter and elliptic orbits of the Asteroid. It turns out that for 
$h\in(-3/2,\sqrt 2)$ there 
are open sets $U_h$ where the orbits of Jupiter and the Asteroid 
intersect, see  
Fig. \ref{i-0}. Due to the nontrivial dependence of the period of the Asteroid with 
respect to the semimajor axis of the associated ellipse, there is a dense subset of 
tori in $U_h$ such that periods of Jupiter and the Asteroid are incommensurable. 
As a result,  collision orbits are dense.  

The proof of Theorem \ref{main}  consists in justifying that similar phenomenon 
takes place for $\mu>0$. In this case there are collisions and 
the  Hamiltonian of the RPC3BP becomes singular. Notice that the collision 
in $\cU$ happens only between Jupiter and the Asteroid,  but not with the Sun. 
The Jupiter-Asteroid collisions were studied by Bolotin and McKay \cite{BM}. 
\begin{rmk}\label{rmk:MainBetterDensity}
 The density exponent in Theorem \ref{main} can be slightly improved from 
$\frac{1}{20}$ to $\frac{1}{17+\nu}$ for any $\nu>0$ by refining the proof. 
See Remarks \ref{rmk:HermanKAM1} and  \ref{rmk:HermanKAM2}. 
\end{rmk}

\begin{rmk}
The results given in the papers  \cite{CL,F}, which study the existence of KAM 
solutions containing collisions also lead to asymptotic density of 
collision orbits result. Nevertheless, those papers only lead to such density 
in very small sets. Let us note that in \cite{F} KAM tori 
passing through a collision can occupy a set of large positive measure provided 
that the distance among bodies is not 
uniformly bounded. 

Theorem \ref{main} gives asymptotic density in a ``big'' set 
independent of $\mu$. In Delaunay variables 
our set $\mathcal U$ is the interior of any compact set contained in 

\begin{equation}\label{eq:our-region} 
\mathcal V= \left\{
-\frac{1}{2L^2}-L\sqrt{1-e^2} \in (-2\sqrt 2,3), \qquad 
L^2(1-e)<1<L^2(1+e)\right\},
\end{equation}
where $0\le e<1$ is the eccentricity and $L^2>0$ is the semimajor 
axis (see Figure \ref{i-0}). In particular, the volume of this set 
can {\it exceed any predetermined constant}, provided that 
$\mu$ is small enough. See section \ref{sec:2BP} for more 
details.

\end{rmk}

With similar techniques, we can disprove a weak 
version of Herman's conjecture. Let us define approximately 
non-wandering points.



\begin{defn}\label{defi:wandering:approx}
Consider a a dynamical system $\{\phi_t\}_{t\in\RR}$ defined on a topological space $X$. 
Then, a point $x\in X$  is called $\delta$-non-wandering, 
if for any neighborhood $\cU$ of it containing 
the $\delta$-ball $B_\delta(x)$, there exists $T>1$
such that $\phi_T(\cU)\cap\cU\neq\emptyset$. \\ 
\end{defn}

\begin{thm}[Second main result]\label{main1}
Any point belonging to the  open set $\cU$ considered in Theorem \ref{main} is  $O(\mu^{\frac{1}{20}})-$non wandering under the flow associated to the Hamiltonian $H_\mu$  in (\ref{r3b-e})

More concretely, for any $z\in \cU$, we can find a
$O(\mu^{\frac{1}{20}})$-neighborhood $\cV_\mu$ of it and 
times  $0<T'_\mu<T_\mu$ such that 
$\phi_{H_\mu}(T'_\mu,\cV_\mu)$ is 
$O(\mu^{\frac{1}{20}})-$close to a collision and  
$\phi_{H_\mu}(T_\mu, \cV_\mu)\cap\cV_\mu\neq\emptyset$. \\
\end{thm}

We devote the main part of this paper to prove Theorem \ref{main}. Then in 
Section \ref{sec:Wandering}, we prove Theorem \ref{main1} by using the partial 
results obtained in Section \ref{sec:R1} to prove Theorem \ref{main}.

\begin{rmk} The existence of 
$O(\mu^{\frac{1}{20}})$-non-wandering sets for the RPC3BP 
is not a new result. In some ``collisionless'' regions of  phase 
space it follows from the KAM Theorem for small $\mu$. 
Theorem \ref{main1} extends such property to a ``collision'' 
region of the phase space $\mathcal U$, see 
\eqref{eq:our-region}. Moreover, we believe that if 
Alekseev conjecture were true, application of our method 
would give a dense wandering set in $\mathcal U$ and 
{\it contradict Herman's conjecture!}
\end{rmk}

We finish this paper by summarizing the scheme and the main heuristic ideas of the proof of Theorem \ref{main}.\medskip 

{
\noindent{\bf Scheme of the proof of Theorem \ref{main}:} For the convenience of a local analysis, we shift the position of Jupiter to zero, i.e. via the transformation 
\[
\Psi_0: u=x-(1-\mu,0),\quad v=y-(0,1-\mu)
\]
the Hamiltonian becomes
\begin{equation}\label{r3bp-local}
H_\mu(u,v)=\frac{|v|^2}{2}-u^tJ v-(1-\mu)u_1-\frac{\mu}{|u|}-\frac{1-\mu}{|u+1|}-\frac1{2}(1-\mu)^2.
\end{equation}
where $(u,v)\in\R^4$. Consider the following division of the phase space:
\begin{equation}\label{def:Regions}
\left\{
\begin{aligned}
R_1&:=\{|u|\geq\mu^{\frac{3}{20}}\},& \text{
Influence of the Sun dominates}\\
R_2&:= \{\rho\mu^{\frac 12}\leq|u|\leq\mu^{\frac{3}{20}}\},  
& \text{
Influence of the Sun \& Jupiter may be comparable}\\
R_3&:=\{0<|u|\leq\rho\mu^{\frac 12}\},\ 
&0\ <\mu\ll\rho\ll1,  \qquad 
\qquad \text{Influence 
of Jupiter dominates}
\end{aligned}
\right.
\end{equation}

\medskip
The proof of Theorem \ref{main} consists of three steps:
\begin{enumerate}
\item {\sf (From global to local)} For sufficiently small $0<\mu\ll1$, and any initial point $\mathbb X\in\cU$, we can find a segment $\mathfrak S$ 
of the length $O(\mu^{\frac {3}{20}})$ and 
$\mathrm{dist}(\mathfrak S,  \mathbb X)\leq O(\mu^{\frac{1}{20}})$ 
in the phase space, such that the push forward of $\mathfrak S$ 
along the flow of $H_\mu$ will become a segment 
\begin{equation}\label{1st segment}
\cS_0\subset \partial (R_2\cup R_3),
\end{equation}
which is a graph over the configuration space so that incoming velocity 
satisfies certain quantitative estimates (see Prop. \ref{prop:r1-region-main}
for more details and Fig. \ref{pc}). Inclusion (\ref{1st segment}) implies that $\cS_0$ lies in the boundary of the local region $R_1^c=R_2\cup R_3$. 
Now we turn to a local analysis summarized on Fig. \ref{lr}.\\

\item {\sf (Transition zone)} In this step we show that there exists
a subsegment $\cS'_0\subset \cS_0$ such that the push forward along 
the flow of $H_\mu$ becomes a segment
 \[
\cS_1\subset \partial R_3
\]
so that the shape of $\cS_1$ and incoming velocity satisfy certain 
quantitative estimates (see Proposition \ref{prop:r2-region}  and 
Fig. \ref{lr} \, for more details). In the region $R_2$, which is 
$\mu^{\frac {3}{20}}$-small we come with velocity $O(1)$ and show that 
linear approximation suffices, even though neither the Sun, nor
Jupiter have dominant effect in this region.\\


\item {\sf (Levi-Civita region and the local manifold of collisions)} 
In the region $R_3$, we can apply the Levi-Civita regularization and deduce 
a new system close to a linear hyperbolic system. We analyze 
the local manifold of collisions, denoted by $\Upsilon$ and show that 
$\cS_1$ intersects $\Upsilon$. This implies the existence of collision orbits 
starting from $\cS_1$, and, therefore, from $\mathfrak S$ (see Lemma \ref{collisional-cover} and Fig. \ref{cross}).\\
\end{enumerate}
}
{
\noindent{\bf Heuristic ideas in the proof:} Here we describe main ideas of the proof:
\begin{itemize}
\item {\sf (From global to local)} In order to control the long time evolution of $\mathfrak S$ 
we apply the following trick: Inside the local region $R_2\cup R_3$, we modify $H_\mu$ into
 $\widehat H_\mu$ by removing the singularity. This enables us to apply the KAM theorem. 
Thus we can pick up a segment $\mathfrak S$ on a suitable KAM torus $\cT_w$ and show
that the push forward along the flow of $H_\mu$ coincides with the flow of $\widehat H_\mu$, 
as long as it does not enter the collision region $R_2\cup R_3$. We also show that the final 
state of $\cS_0$ is a graph over the configuration space with almost constant velocity 
component. More precisely, for any point in $\cS_0$, the velocity is contained in 
a $O(\mu^{\frac {3}{20}})$ neighbourhood of a certain velocity $v_0$ 
(see Prop. \ref{prop:r1-region-main} and Fig. \ref{pc} for more details). \\

\item {\sf (Transition zone)} We start with the curve $\cS_0$, which has 
almost constant velocity. Then we flow the segment by the flow of $H_\mu$ 
using that it is close to linear. Controlling the evolution of the flow we get 
the desired estimate on the final state $\cS_1$ of $\cS_0$ 
(see Proposition \ref{prop:r2-region} and Fig.  \ref{lr}).\\


\item {\sf (Levi-Civita region and local collision manifold)} Once we have information about  $\cS_1$, the approximation by the linear hyperbolic system gives precise enough local information about the collisions manifold $\Upsilon$. This allows us to prove that $\cS_1\bigcap\Upsilon\neq\emptyset$ by using the intermediate value theorem  (see Lemma \ref{collisional-cover} and Fig. \ref{cross}). \\
\end{itemize}
}

\noindent{\bf Organization of this paper:}  The paper is 
organized as follows. In Section \ref{sec:2BP}, we introduce 
the  Delaunay coordinates and discuss the integrable 
Hamiltonian \eqref{r3b-e} with $\mu=0$. In Section \ref{sec:R1}, 
we analyze the dynamics ``far away'' from collisions (Step 1 of 
the Scheme of the proof). We  define the modified
Hamiltonian $\widehat H_\mu$ and we  apply the KAM theory. 
Then in Section \ref{sec:transition} we analyze the dynamics 
in the transition zone (Step 2).
In Section \ref{sec:LeviCivita}, we use 
the Levi--Civita regularization to analyze a small
neighborhood of the collision (Step 4). This completes 
the proof of Theorem \ref{main}. Finally, in the Appendix 
we provide basic formulas for Delaunay coordinates.\\

\noindent{\bf Acknowledgments:} The authors thank Alain Albouy,  Alain 
Chenciner and  Jacques F\'ejoz  for helpful discussions  and  remarks  on  a  
preliminary  version  of  the  paper.

\section{The collision set and density of collision orbits for $\mu=0$}\label{sec:2BP}

We start by considering Hamiltonian \eqref{r3b-e} with $\mu=0$. This simplified
model will give us the open set $\cV$ where to look for (asymptotic) density of
collisions. The analysis of this set was already done in \cite{BM}. 
Hamiltonian \eqref{r3b-e} with $\mu=0$ reads
\begin{equation}\label{r3b-0}
H_0(x,y)=\frac{|y|^2}{2}- x^t Jy-\frac{1}{|x|}.
\end{equation}
If we perform the classical Delaunay transformation (see
Appendix \ref{app:Delaunay})
$\Psi(x,y)=(\ell,g,L,G),$
which is symplectic,  to $H_0$ we obtain
\begin{equation}\label{0-r-a-a}
H_0(L,G)=-\frac{1}{2L^2}-G.
\end{equation}
We use these coordinates to define the set $\cV$ where collisions orbits are dense when $\mu=0$. We define also the eccentricity 
\begin{equation}\label{def:eccentricity}
e=e(L,G)=\sqrt{1-\frac{G^2}{L^2}}.
\end{equation}

\begin{figure}
\begin{center}
\includegraphics[width=6cm]{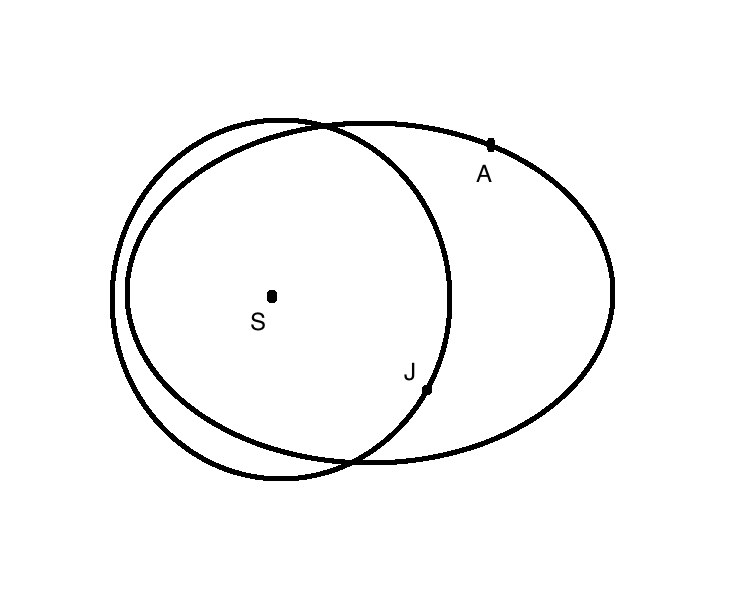}
\caption{Elliptic and circular orbits of Asteroid and Jupiter resp. for $\mu=0$.}
\label{i-0}
\end{center}
\end{figure}

\begin{lem}[\cite{BM}]\label{lemma:2BPdensity}
Fix $J\in(-2\sqrt 2,3)$ and define the open set 
\[
 \cV=\left\{(\ell,g,L,G)\in \T^2\times(0,+\infty)\times (-L, 0)\cup (0,L): \ \frac{G^2}{1+e}<1<\frac{G^2}{1-e}
 \right\}.
\]
Then, the set 
\[
 \cV_J=\cV\cap \{-2H_0=J\}
\]
contains a dense subset whose orbits  tend to collision.
\end{lem}

\begin{proof}  
To prove this lemma, we express the collision set in  Delaunay coordinates (see Appendix \ref{app:Delaunay}). 
This expression is needed in Section \ref{sec:R1}.
In polar coordinates the collisions are defined (when $\mu=0$) by 
\[
r=1,\quad \varphi=0.
\]
By  \eqref{def:rphiDelaunay}, this is
equivalent to 
\begin{equation}\label{def:CollisionDelaunay}
\begin{split}
L^2(1-e\cos \mathfrak{u})&=1\\
\mathfrak{v}(\ell)+g&=0.
\end{split}
\end{equation}
To have solutions of the first equation, we impose
\begin{equation}\label{def:condLe}
\Big|\frac{L^2-1}{eL^2}\Big|<1,
\end{equation}
which is equivalent to the condition
\begin{equation}\label{cond-3}
\frac{G^2}{1+e}<1<\frac{G^2}{1-e},
\end{equation}
imposed in the definition of $\cV$.
Assuming this condition,
the first equation has  two solutions in $[0,2\pi]$
\[
\mathfrak u^*_+=\arccos{\frac{L^2-1}{eL^2}}\in(0,\pi),\quad \mathfrak u_-^*=2\pi-\arccos{\frac{L^2-1}{eL^2}}\in(\pi,2\pi).
\]
Using $\ell=\mathfrak{u}-e\sin\mathfrak{u}$,  we obtain
$\ell_{\pm,0}^*(L,G)$. Finally,  we can solve the second equation in
\eqref{def:CollisionDelaunay} as
$g_{\pm,0}^*(L,G)=-\mathfrak v(\ell^*_\pm)$ to obtain the collision set as two 
graphs  on the actions $(L,G)$,
\begin{equation}\label{def:CollisionGraph2bp}
\begin{split}
\ell&=\ell^{\pm,0}_\col(L,G)\\
g&=g^{\pm,0}_\col(L,G).
\end{split}
\end{equation}

Recall that $H_0(L,G)$ is completely integrable. For fixed $J\in(-2\sqrt
2,3)$, 
$$
\cV\cap\{-2H_0(L,G)=J\}
$$ 
is foliated by 2-dimensional tori defined by constant $(L,G)$ 
(see Fig. \ref{actionangle}), whose dynamics is a rigid rotation with 
frequency vector $\omega=(\partial_L H_0, -1)$. If
$\partial_L H_0=L^{-3}\in\R\backslash\Q$, the orbit 
\[\left\{\left.\varphi_t\left(\ell^{\pm,0}_\col(L,G),g^{\pm,0}_\col(L,G),L,\frac{J}
{2} -\frac{1}{2L^2}\right)\ \right|\ t\in\R\right\}\] is dense in 
the corresponding torus.  Moreover,
$\partial_L
H_0={1}/{L^3}$ is a  diffeomorphism of $(0,+\infty)$. Thus, for a dense set
$L\in(0,+\infty)$, the frequency vector is non-resonant. These two facts lead to density of collisions lead to the
existence of $\cV$ of which collision solutions are dense.
\end{proof}
\medskip

 Lemma \ref{lemma:2BPdensity} does not only provide the open set $\cV$ but also describes it in terms of the  Delaunay coordinates.  Let us explain the set $\cV$ geometrically. 
We need to avoid the following:
\begin{itemize}
 \item Degenerate ellipses with $e=1$: so we impose $G\neq 0$.
 \item Circles: so we impose $|G|< L$.
 \item Ellipses that do not intersect the orbit of the second primary (the unit 
circle) or are tangent to it. This is given by two conditions. The first one is 
\eqref{cond-3}. The second one is that the semimajor axis $L$ cannot be too 
small. This second condition is equivalent to take $H_0$ in the imposed range of 
energies $-2H_0=J\in(-2\sqrt 2,3)$.\\
\end{itemize}

The proof of Lemma \ref{lemma:2BPdensity} also provides a description of the collision manifold for $H_0$ in $\cV\cap\{-2H_0(L,G)=J\}$. This manifold has two connected components in the energy level defined as 
\[
\cC_J^\pm=\Big\{(\ell,g,L,G)\in \cV\cap\{-2H_0(L,G)=J\}\Big|\ell=\ell_{\pm,0}^*(L,G),g=g_{\pm,0}^*(L,G)\Big\}
\]
It can be easily seen that these manifolds intersect transversally each 
invariant torus  $(L,G)=\mathrm{constant}$ in $\cV\cap\{-2H_0(L,G)=J\}$.

Finally, let us point out that to prove Theorem \ref{main} we cannot work 
with the full set $\cV$ but in open sets whose closure is strictly contained 
in $\cV$. Namely, the closer we are to the boundary of $\cU$, the smaller 
we need to take $\mu$ to prove Theorem \ref{main}. To this end, we define 
the following open sets. Fix $\de>0$ small. 
Recall that eccentricity $e=e(L,G)=\sqrt{1-\frac{G^2}{L^2}}$, see  \eqref{def:eccentricity}.
Then, we define 
\[
 \cV_\de\ \subset\ \overline{\cV_\de}\ \subset\ \cV,
\]
where 
\begin{equation}\label{def:Vdelta}
\beal 
\cV_\de= \left\{(\ell,g,L,G)\right. & \in \cV:  \ \ L\in (\de,\de^{-1}), \\ 
& \left. \ \de<|G|<L- \de,  \ \   \frac{G^2}{1+e(G,L)}+\de<1<\frac{G^2}{1-e(G,L)}-\de
 \right\}.
\enal 
\end{equation}

For $\mu>0$, one can analyze the collision set analogously as done 
in the proof of Lemma \ref{lemma:2BPdensity}. One just 
needs to replace the equations \eqref{def:CollisionDelaunay} by
\begin{equation}\label{def:CollisionDelaunay:mu}
\begin{split}
L^2(1-e\cos \mathfrak{u})&=1-\mu\\
\mathfrak{v}(\ell)+g&=0.
\end{split}
\end{equation}
which have solutions in $\mathcal V_\de$ for $\mu$ small enough and  lead to a 
definition of the collision set as two graphs
\begin{equation}\label{def:CollisionSet:mu}
\begin{split}
\ell&=\ell^{\pm,\mu}_\col(L,G)\\
g&=g^{\pm,\mu}_\col(L,G).
\end{split}
\end{equation}
Moreover, these graphs are non-degenerate in 
$\mathcal V_\de$ as the associated Hessian has positive lower 
bounds (independent of $\mu$).\\ 

\begin{figure}
\begin{center}
\includegraphics[width=8cm]{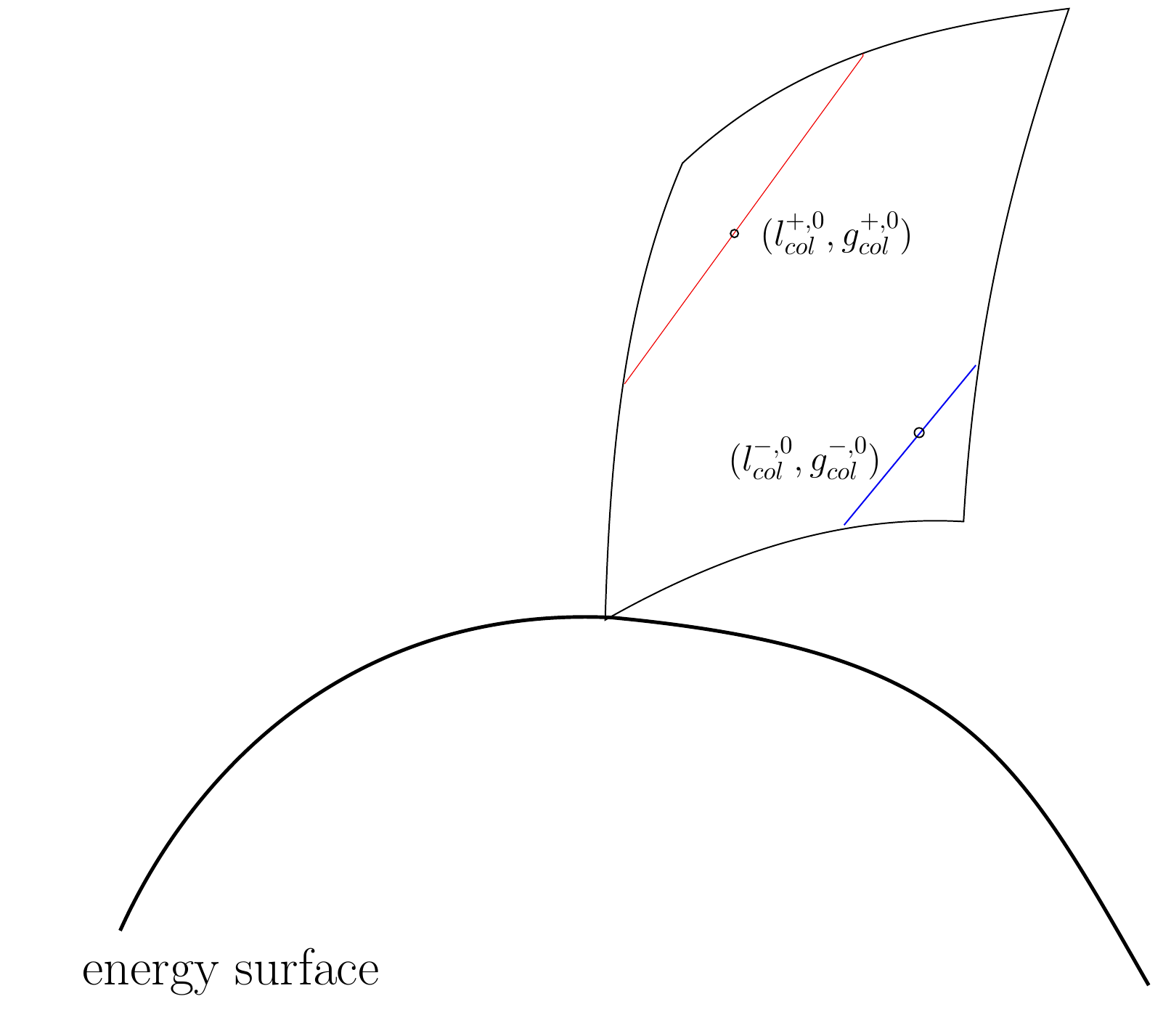}
\caption{For $\mu=0$, the energy surface $\{-2H_0=J\}$ is foliated by punctured tori, where the punctures correspond to collisions.}
\label{actionangle}
\end{center}
\end{figure}

\section{The region $R_1$: dynamics far from collision}\label{sec:R1}

To study the region $R_1$, that is dynamics ``far from collision'',  
we apply KAM Theory. To this end, we modify  the Hamiltonian to avoid 
its blow up when approaching collision.  We modify 
the Hamiltonian in polar coordinates and then we express the modified Hamiltonian in Delaunay variables.

The Hamiltonian \eqref{r3b-e} expressed in polar coordinates \eqref{def:polars} is given by 
\begin{eqnarray}\label{def:HamPolars}
H_\mu(r,\varphi,R,G)&=&\frac{R^2}{2}+\frac{G^2}{2r^2}-G-\frac{\mu}{\sqrt{r^2+(1-\mu)^2-2(1-\mu)r\cos\varphi}}\nonumber\\
& &-\frac{1-\mu}{\sqrt{r^2+\mu^2+2\mu r\cos\varphi}}
\end{eqnarray}
which can be written as 
\[
H_\mu(r,\varphi,R,G)=\frac{R^2}{2}+
\frac{G^2}{2r^2}-G-\frac 1r-\mu
g_1(r,\varphi,\mu)-\mu g_2(r,\varphi,\mu),
\]
where 
\[
 \begin{split}
  g_1(r,\varphi,\mu)=&\frac{1}{\sqrt{r^2+(1-\mu)^2-2(1-\mu)r\cos\varphi}}\\
  g_2(r,\varphi,\mu)=&\mu^{-1}\left(\frac{1}{\sqrt{r^2+\mu^2+2\mu
r\cos\varphi}}-\frac{1}{r}\right).
\end{split}
 \]
The term $g_1$ has a singularity at $\{(r,\varphi)=(1-\mu,0)\}$ and $g_2$ 
is analytic in the domains we are considering (which do not 
contain the position of the other primary). We  modify $g_1$ by
multiplying it by a $C^\infty$ smooth bump function. Consider 
$\Phi:\mathbb R\to \mathbb R$ so that 
\[
 \Phi(z)=\left\{\begin{matrix} 
0 \ \text{ if }\ |z|\leq 1\\ 
1 \ \text{ if }\ |z|\geq
2\end{matrix} \right.
\]
Then, if we fix $\tau>0$, we define
\[
 \widehat{g}_1(r,\varphi,\mu)=
\Phi\left(\mu^{-\tau}\sqrt{
(r\cos\varphi-1+\mu)^2+r^2\sin^2\varphi } \right)
\left( { g } _1(r , \varphi ,
\mu)-4\mu^{-\tau}\right)+4\mu^{-\tau}.
\]
with
\[
\widehat{g}_1(r,\varphi,\mu)=\left\{
\begin{split}
g_1(r,\varphi,\mu),&\quad\text{for }|(r\cos\varphi-1-\mu,r\sin\varphi)|\geq 2\mu^\tau\\
4\mu^{-\tau},&\quad\text{for }|(r\cos\varphi-1-\mu,r\sin\varphi)|\leq \mu^\tau.\\
\end{split}
\right.
\]
Later, in Section \ref{sec:DensityDelaunay}, we show that the optimal choice for $\tau$ is $\tau=3/20$. \\

Notice that $\|\widehat g_1\|_{C^r}\lesssim\mu^{-(r+1)\tau}$ for sufficiently small
$\mu\ll1$, and $\|g_2\|_{C^r}\lesssim 1$. In this section, we consider  the
modified Hamiltonian 
\begin{equation}\label{modi-h}
\widehat{H}_\mu(r,\varphi,R,G)=\frac{R^2}{2}+\frac{G^2}{2r^2}
-G-\frac 1r-\mu \widehat{g}_1(r,\varphi,\mu)+
\mu g_2(r,\varphi,\mu),
\end{equation}
and we express it in Delaunay coordinates by considering the transformation 
$\Psi_2(r,\varphi,R,G)=(\ell,g,L,G)$ introduced in \eqref{trans}. This change 
leads to an iso-energetic non-degenerate nearly integrable Hamiltonian
\begin{equation}\label{i-i-h}
\widehat H_\mu(\ell,g,L,G)=-\frac{1}{2L^2}-G+\mu\widehat{f}_1(\ell,g,L,G,\mu)-\mu
f_2(\ell,g,L,G,\mu).
\end{equation}
Fix $\de>0$. Then, in the set $\cV_\de$ defined in \eqref{def:Vdelta}, the functions $f_1$ and $f_2$ satisfy
\[
\|\widehat f_1\|_{C^r}\leq C\mu^{-(r+1)\tau},\quad \|f_2\|_{C^r}\leq C
\]
for some constant $C$ which depends on $\de$ but is independent of $\mu$.

In  polar coordinates, there are two disjoint subsets
\begin{equation}\label{def:PuncturePolars}
\cD_{\pol}^\pm:=\Big\{(r,\varphi,R,G)\subset\Psi_2(\cV_\de)\ \Big|\ \left|
 (r\cos\varphi-1+\mu,r\sin\varphi)\right|\leq \mu^\tau\Big\},
\end{equation}
at each of the considered energy levels where 
the Hamiltonian $H_\mu$ in \eqref{def:HamPolars} is 
different from the modified $\widehat H_\mu$ in 
\eqref{modi-h}. They correspond to two disjoint 
intersections (see Fig. \ref{i-0}). Here the sign $\pm$ 
depends on the sign of the variable $R$.

The main result of this section is the following Proposition, 
where we take 
\[
 \tau=\frac{3}{20}.
\]
Note that we abuse notation and we refer to $\cV_\de$ independently of the coordinates we are using.
\begin{prop}\label{prop:r1-region-main}
Fix $\de>0$ and  $\varpi>0$ small. Then there exists  
$\mu_0>0$ depending on $\de$ and $\varpi$, such that 
the following holds for any $\mu\in (0,\mu_0)$:

For any ${\mathbb  X}\in\cV_\de$, there exists a $C^1$ curve 
$\mathfrak S\in\cV_\de$  of  length $O(\mu^\frac{3}{20})$ satisfying  
\[\mathrm{dist}(\mathfrak S, \mathbb X)\leq O(\mu^\frac{1}{20})\]
and a continuous function  $T_0:\mathfrak S \to\mathbb{R}^+$
such that 
\[
\cS_0=\left\{\phi_{H_\mu}(T_0(z),z):z\in\mathfrak S\right\}
\]
satisfies either  $\cS_0\subset \partial \cD_{\pol}^+$ or $\cS_0\subset 
\partial \cD_{\pol}^-$, where $\phi_{H_\mu}$ is the flow associated to the Hamiltonian $H_\mu$.

Moreover, 
\begin{enumerate}
\item  There exists a $C^1$ function $V$  satisfying
such that $\cS_0$ is a graph over $u$ as
\[\cS_{0}=\left\{(u,V(u))\ \left|\ \ \ u=\mu^{\frac{3}{20}} e^{is}\cdot 
\frac{v_{0}}{|v_{0}|}, \quad s\in\left[\frac\pi 
2+\varpi,\frac{3\pi}{2}-\varpi\right]\right.\right\}.\]
Moreover, there exists $v_0\in\mathbb R^2$ satisfying $|v_0|\geq C$ for certain $C>0$ 
independent of $\mu$ such that
\[\max\, |V(u)-v_{0}|\leq  O(\mu^{1/20}).\]
\item For all $z\in \mathfrak S$ and $t\in (0, T_0(z))$, $\phi_{\widehat H_\mu}(t,z)\not\in \cD_{\pol}^+\cup  \cD_{\pol}^-$  and, therefore, 
\[\phi_{\widehat H_\mu}(t,z)=\phi_{H_\mu}(t,z),\quad \forall z\in \mathfrak S \text{ and }t\in(0,T_0(z)).\]
\end{enumerate}
\end{prop}

This proposition implies that any point in $\cV_\de$ has a curve 
$\mathfrak S$ in its $O(\mu^\frac{1}{20})$ neighborhood that hits 
``in a good way'' a $O(\mu^\frac{3}{20})$ 
of the collision. To prove 
Theorem \ref{main}, it only remains to prove that the image curve $\cS_0$ 
posesses a point whose orbit leads to collision. We prove this fact in 
two steps in Sections \ref{sec:transition} and \ref{sec:LeviCivita}.

The rest of this section is devoted to prove Proposition \ref{prop:r1-region-main}. 

\begin{proof}
The proof has several steps. We first analyze the dynamics in 
the region $R_1$ in Delaunay coordinates, then translate into 
the Cartesian coordinates $(u,v)$.

\subsection{Application of the KAM Theorem}

 First step is to apply KAM Theorem to get invariant tori for the Hamiltonian 
$\wh H_\mu$. We are not aware of any KAM Theorem in the literature dealing 
with $C^{\infty}$ iso-energetically non-degenerate Hamiltonian systems. 
To overcome this problem, we reduce $\wh H_\mu$ to a two dimensional 
Poincar\'e map and use Herman's KAM Theorem \cite{H}.

\begin{lem}\label{lemma:PoincareMap}
Fix $r\geq 3$ and $\tau>0$ such that $1-(r+2)\tau>0$. Consider the Hamiltonian \eqref{i-i-h} and fix an energy level
$\{\widehat{H}_\mu=h\},\ h \in (-3/2,\sqrt{2})$. Then, for $\mu$ small enough, the flow
associated to \eqref{i-i-h} restricted to the level of energy induces a two
dimensional exact symplectic  Poincar\'e map
$\mathcal P_{h,g_0}:\{g=g_0\}\to\{g=g_0\}$, $\mathcal P_{h,g_0}=\mathcal
P_{h,g_0}(\ell, L)$.
Moreover, $\mathcal P_{h,g_0}$ is of the form 
\[
 \mathcal P_{h,g_0}:\begin{pmatrix}\ell\\L\end{pmatrix}\to
\begin{pmatrix}\ell-2\pi \omega
(L)\\L\end{pmatrix}+F\begin{pmatrix}\ell\\L\end{pmatrix}
\]
where
\[
 \omega(L)=\frac{1}{L^3}
\]
and  $F$ depends on both
$h$ and $g_0$ and satisfies
\[
\|F\|_{C^r}\leq
C\mu^{1-(r+2)\tau}
\]
for some $C>0$ independent of $\mu$.
\end{lem}
We apply KAM Theory to the Poincar\'e map $\mathcal P_{h,g_0}$. 
Recall that a real number $\omega$ is called 
a  {\it constant type  Diophantine number} if there
exists a constant $\gamma > 0$ such that 
\begin{equation}\label{eq:dioph-gamma}
\left|\omega-\frac{p}{q}\right| \geq \frac{\gamma}{q^{2}}\qquad\text{
for all }p \in \mathbb Z,\ \  q \in \mathbb N.
\end{equation}
We denote  by  $B_\gamma$ the set of such numbers for a fixed $\gamma>0$. The 
set $B_\gamma$ has measure zero. Nevertheless, it has the following property.
\begin{lem}\label{lemma:densityconstantnumbers}
Fix $\gamma\ll1 $. Then, the set $B_\gamma$ is 
$\gamma$-dense in $\R$ 
\end{lem}
We prove this lemma in  Appendix \ref{app:Diophantine}.

Then we can apply the following KAM theorem.
\begin{thm}[M. Herman \cite{H}, Volume 1, Section 5.4 and 5.5]\label{kam}
Consider a $C^r$, $r\geq 4$, area preserving twist map 
\[
f_\eps:[0,1]\times\T\rightarrow [0,1]\times\T\quad \text{of the form}\quad 
f_\eps=f_0+\eps f_1,
\] 
where 
\[
f_0(\theta,I)=(\theta+A(I),I)
\]
and $M_0^{-1}\geq A'(I)\geq M_0>0$ for all $I\in\R$. Assume  $\|f_1\|_{C^r}\lesssim 1$. 
Then, if $\eps^{1/2}M_0^{-1}$ is small enough,
for each $\omega$ from the set of constant type Diophantine numbers with $\gamma\sim \eps^{1/2}$, the map $f_\eps$ posesses an invariant torus $\mathcal T_\om$ which is a graph of $C^{r-3}$ 
functions $U_\omega$ and the motion on $\mathcal T_\om$ is 
$C^{r-3}$ conjugated to a rotation by $\omega$ with
$\|U_\omega\|_{C^{r-3}}\lesssim\eps^{1/2}$. These tori cover 
the whole annulus $O(\eps^{1/2})$-densely.
\end{thm}

\begin{rmk}\label{rmk:HermanKAM1}
In  \cite{H} it is shown that this theorem is also valid under the weaker assumption that the map $f_\eps$ is $C^{3+\beta}$ with any $\beta>0$ instead of $C^4$. This would slightly improve the density exponent in Theorem \ref{main} as already pointed out in Remark \ref{rmk:MainBetterDensity} (see also the Remark  \ref{rmk:HermanKAM2} below). We stay with regularity $C^4$ to have simpler estimates.
\end{rmk}

%

This theorem can be applied to the Poincar\'e map obtained in Lemma
\ref{lemma:PoincareMap}. 
Moreover, these KAM tori have smooth dependence on $g_0$. 
Indeed, all Poincare maps $\mathcal P_{h,g_0}:\{g=g_0\}\to\{g=g_0\}$ 
with different $g_0$ are conjugate to each other. 

Theorem \ref{kam} implies 
the existence of 2--dimensional tori $\cT_\omega^h$ which are invariant 
by the flow of $\widehat H_\mu$ in \eqref{i-i-h} with  energy 
$h=-J/2 \in (-3/2,\sqrt{2})$. Note that we cannot identify 
the quasiperiodic frequency $\omega=(\omega_\ell^h,\omega_g^h)$ 
of the dynamics  on $\cT_\omega$, only that their ratio 
$\omega_\ell^h/\omega_g^h=-1/L_{0,\omega}^3$ is fixed (and Diophantine).

\begin{cor}\label{corollary:KAM:eps-close}
For each $\widehat\omega\in B_\gamma$ with   $\gamma$ satisfying 
$\gamma\sim\eps^{1/2}$ and any $h\in(-3/2,\sqrt2)$ fixed, 
there is a KAM torus $\cT^h_\omega$, which is given by 
\[\cT_\omega^h=\{(\ell,g,L_{\omega,\mu}^h(\ell,g),G_{\omega,\mu}^h(\ell,g))\ 
|\ (\ell,g)\in\T^2\}\]
where $\omega_2/\omega_1=\widehat\omega$ and $(L_{\omega,\mu}^h,G_{\omega,\mu}^h)$ is a $C^{r-3}$ graph satisfying
\begin{equation}\label{drift}
\|L_{\omega,\mu}^h-L_{\omega,0}\|_{C^{r-3}}\lesssim\eps^{1/2},\quad\|G_{\omega,\mu}^h-G_{\omega,0}\|_{C^{r-3}}\lesssim\eps^{1/2}
\end{equation}
where $\eps=\mu^{1-6\tau}$,
\[
\widehat\omega=\frac{1}{L_{0,\omega}^3},\quad h=-\frac{1}{2L_{0,\omega}^2}-G_{0,\omega}.
\]
Moreover, $\bigcup_{\omega\in B^\gamma}\cT_\omega^h$ is $O(\gamma)$-dense in 
$\cV_\de$.
\end{cor}

This corollary is a direct application of Theorem \ref{kam}. 
The frequency in this setting is  given by $\omega(L)={1}/L^3$ and, 
thus
\[
\left|\omega'(L)\right|=\frac{3}{L^4}
\]
has a lower bound independent of $\mu$  (but depending on $\de$) in $\mathcal V_\delta$.  
Since the lower is regularity, the better are estimates for $\eps$ 
we choose $r=4$. To simplify notation, we omit the superindex $h$. Note that the density of the KAM tori is due to the $\gamma$-density of $B_\gamma$, the relation between $\widehat\omega$ and $L$ and \eqref{drift}.

\begin{rmk}
\label{rmk:gammachoice}
Note that one can  apply Theorem \ref{kam} with any 
$\gamma\gtrsim \sqrt{\eps}$ at the expense of obtaining a worse density of invariant tori. 
In Section \ref{sec:DensityDelaunay}, we choose $\gamma$ to optimize density for the collision orbits.
\end{rmk}

\subsection{The segment density argument in Delaunay 
coordinates}\label{sec:DensityDelaunay}
We use the KAM Theorem  to obtain the segment density estimates stated
in Proposition \ref{prop:r1-region-main}. We first obtain this density result 
in Delaunay coordinates. Taking into account that the change from Delaunay 
to the Cartesian coordinates $(u,v)$ is a diffeomorphism with uniform bounds 
independent of $\mu$, this will lead to the density estimates in Proposition 
\ref{prop:r1-region-main}.

For $\mu=0$ Lemma \ref{lemma:2BPdensity} describes the collision set 
in Delaunay coordinates as (two) graphs over the actions $(L,G)$ (see 
\eqref{def:CollisionGraph2bp}). By the implicit function theorem 
the same holds for small $\mu>0$ (see \eqref{def:CollisionDelaunay:mu}). 
Since the KAM tori obtained in Corollary \ref{corollary:KAM:eps-close} 
are graphs over $(\ell, g)$ and ``almost horizontal'' (see \eqref{drift}), 
the intersection between each of these KAM tori $\cT$ and 
the collision set consist of two points $(\ell_\col^{\pm,\mu}, 
g_\col^{\pm,\mu}, L(\ell_\col^{\pm,\mu}, g_\col^{\pm,\mu}), 
G(\ell_\col^{\pm,\mu}, g_\col^{\pm,\mu}))$. 
Denote the restriction of the collision neighborhoods 
$\cD^\pm_{\pol}$ to these cylinders by $\cD^\pm$.
Since the coordinate change from the polar coordinates to 
Delaunay is a diffeomorphism there are constants $C>C'>0$ 
independent of $\mu$ such that 
\begin{equation}\label{eq:LocalizationDpol}
\partial 
\cD^\pm\subset \left\{C'\mu^\tau\leq |(\ell-\ell^{\pm,\mu}_\col, g-g_\col^{\pm,\mu})|\leq 
C\mu^\tau\right\}.
\end{equation}

For any of the  tori $\mathcal T$ obtained in Corollary 
\ref{corollary:KAM:eps-close} we consider their graph parameterization
\[
\cT=\{(\ell,g,L_{\omega,\mu}^h(\ell,g),G_{\omega,\mu}^h(\ell,g))|(\ell,
g)\in\T^2\}
\]
and we  define the balls
\begin{equation}\label{def:PunctureInTorus}
\mathcal B_{\mathcal T}^\pm=\mathcal T\cap \left\{|(\ell-\ell^{\pm,\mu}_\col, 
g-g_\col^{\pm,\mu})|\leq C\mu^\tau\right\}.
\end{equation}

These balls can be viewed on Fig. \ref{actionangle} as neighborhoods of 
marked collision points in each torus. 
The main result of this section is the following 
lemma.

\begin{lem}\label{lemma:DensityDelaunay}
Fix $\de>0$ and  $\varpi>0$ small. Then, there exists  $\mu_0>0$ 
depending on $\de$ and $\varpi$, such that the following holds 
for any $\mu\in (0,\mu_0)$:

For any ${\mathbb  X}\in\cV_\de$, there exists a invariant torus $\cT$ 
obtained in Corollary \ref{corollary:KAM:eps-close} and 
a  $C^1$ curve $\mathfrak S\subset\cT$  of  length 
$O(\mu^{\frac{3}{20}})$ satisfying  
$\mathrm{dist}(\mathfrak S, \mathbb X)\leq O(\mu^{\frac{1}{20}})$ 
and a continuous function  $T_0:\mathfrak S \to\mathbb{R}^+$
such that 
\[
\cS_0=\left\{\phi_{H_\mu}(T_0(z),z):z\in\mathfrak S\right\}
\]
satisfies either  $\cS_0\in \partial \mathcal B_{\mathcal T}^+$ or 
$\cS_0\in \partial \mathcal B_{\mathcal T}^-$. In addition, 
\begin{enumerate}
\item The set $\cS_0$ is a graph over $(\ell, g)$ and satisfies either
\[
\begin{split}
\cS_{0}=\Big\{(\ell,g,L_{\omega,\mu}^h(\ell,g),G_{\omega,\mu}^h(\ell,
g))\Big|&(\ell,g)=(\ell^{+,\mu}_\col,g^{+,\mu}_\col)+
\mu^{ \frac{3}{20}} e^{is}\cdot 
\frac{\om}{|\om|},\\
&s\in\left[\frac\pi 
2+\varpi,\frac{3\pi}{2}-\varpi\right]\Big\}.
\end{split}
\]
or the same for the collision $(\ell^{-,\mu}_\col,g^{-,\mu}_\col)$.

\item 
For all $z\in \mathfrak S$ and $t\in (0, T_0(z))$, 
$\phi_{\widehat  H_\mu}(t,z)\not\in 
\mathcal B_{\mathcal T}^+\cup \mathcal B_{\mathcal T}^-$  
and therefore 
\[
\phi_{\widehat H_\mu}(t,z)=\phi_{H_\mu}(t,z),\quad \forall 
z\in \mathfrak S \ 
\text{ and }\ t\in(0,T_0(z)).
\]
 \end{enumerate}
\end{lem}

We devote the rest of the section to prove this lemma. Since the segments 
$\mathfrak{S}$ considered are contained in the KAM tori from  
Corollary \ref{corollary:KAM:eps-close},  we will use the density of tori 
to ensure that any point in $\mathcal V_\de$ has one of those segments 
nearby. Thus, we need to ensure that \\ 
\begin{itemize}
 \item[1.] By adjusting $\gamma$ in Corollary \ref{corollary:KAM:eps-close}:  
the KAM tori are dense enough (see Remark \ref{rmk:gammachoice}). \medskip
 \item[2.] There are segments whose \text{future evolution ``spreads densely enough''} 
 on these tori. \medskip
\end{itemize} 

Item 2 requires strong (Diophantine) properties on the frequency of the torus. 
The stronger the conditions we impose on the frequency, the better the 
spreading at expense of having fewer tori. This would give worse density 
in item 1. Thus, we need to obtain a balance between  the density of tori 
in the phase space and the good spreading of orbits in the chosen tori.

Fix one torus $\mathcal T$ from Corollary \ref{corollary:KAM:eps-close} 
and consider the associated balls $\mathcal B_{\mathcal T}^\pm$ given 
by  \eqref{def:PunctureInTorus}. To obtain the density statement, 
we first prove it for points belonging to the torus $\mathcal {T}$. 
Then due to sufficient density of KAM tori, we deduce Lemma \ref{lemma:DensityDelaunay} .

We want to show that any point $z\in \cT$ has a segment 
$\mathfrak S\subset\cT$ in its  $O(\mu^\frac{1}{20})$ neighborhood  which, under the flow of 
Hamiltonian \eqref{r3b-e} (in Delaunay coordinates), hits 
``in a good way'' either  $\partial \mathcal B_{\mathcal T}^-$ or 
$\partial \mathcal B_{\mathcal T}^+$, namely,  covering a large 
enough part of the boundary of the balls and incoming velocity
being almost constant (see Fig. \ref{pc}). 
Note that we apply KAM to the Hamiltonian \eqref{i-i-h} instead of 
the original one \eqref{r3b-e}. 
Since the Hamiltonian coincide only away from the union 
$\mathcal B_{\mathcal T}^- \cup \mathcal B_{\mathcal T}^+$,
we need to make sure that the evolution of $\mathfrak S$
does not intersect this union before hitting it ``in a good way''.
\\ 

To start, assume that $\cT$ has only one collision instead.  
Making a translation, we can assume that it is located at 
$(\ell,g)=(0,0)$. Later, we
adapt the construction to deal with tori having two collisions. \\ 

\noindent  {\bf One collision model case:} \  Since $\cT$ is a graph on $(\ell, g)$, 
we analyze the density in the projection onto the base. 
By Theorem \ref{kam}, the torus and its dynamics are 
$\eps^{1/2}=\mu^{(1-6\tau)/2}$-close to the unperturbed one. Moreover, after a $\eps^{1/2}$-close to the identity transformation, the base dynamics is a rigid rotation.  
Somewhat abusing notation we still denote transformed variables 
$(\ell,g)$. We analyze the density on the section $\{g=0\}$. Since 
the dynamics is a rigid rotation, the density in the section implies 
the density in the whole torus. 

We flow backward the collision and analyze the intersections of the orbits 
with $\{g=0\}$.  By a change of time, the orbits on the projection are just 
\begin{equation}\label{def:LinearFlow}
 (\ell(t),g(t))=(\ell^0+\omega t, g^0+t),
\end{equation}
where $\omega\in B_\gamma$, defined in \eqref{eq:dioph-gamma}, 
with $\gamma\gtrsim \sqrt{\eps}$. The intersections of the backward orbit 
starting at the collision $(0,0)$ with $\{g=0\}$ are given by $\|q\omega\|$, where 
\begin{equation}\label{def:norm}
 \|\alpha\|=\min_{p\in\mathbb Z}|\al-p|.
\end{equation}
Fix $C>0$. We study this orbit until it hits again 
a $C\mu^\tau$ neighborhood of the collision. 
Thus, we consider 
$q=-1,\ldots,-q^*$ where $q^*+1\in\mathbb N$ 
is the smallest solution to 
\[
 \|(q^*+1)\omega\|\leq 4C\mu^\tau.
\]
Assume that the (ratio of) frequencies of the  torus $\cT$ is in  
$B_\gamma$ (with $\gamma$ to be specified later). Then, 
we obtain that 
\begin{equation}\label{def:MaximalTime}
 |q^*|\geq \frac{1}{4C}\gamma\mu^{-\tau}-1.
\end{equation}
We need to study the density of $-q\omega \,\,(\mathrm{mod }\, 1)$ with 
$q=-1,\ldots,- q^*$. We apply the following non-homogeneous Dirichlet 
Theorem (see \cite{Cas}), where we use the notation \eqref{def:norm}.

\begin{thm}\label{thm:Cassels}
Let $L(x)$, $x=(x_1,\ldots, x_n)$ be a linear form and fix $A,X>0$. 
Suppose that there does not exist any $x\in \ZZ^n\setminus 0$ such that 
\[
 \|L(x)\|\leq A\,\,\,\text{ and }\,\,\,|x_i|\leq X.
\]
Then, for any $a\in\RR$, the equations
\[
 \|L(x)-a\|\leq A_1\,\,\,\text{ and }\,\,\,|x_i|\leq X_1.
\]
have an integer solution, where
\[
 A_1=\frac{1}{2}(h+1)A,\ \ \ X_1=\frac{1}{2}(h+1)X
\,\,\textup{ and }\,\,h=X^{-n}A^{-1}.
\]
\end{thm}

We use this theorem to show that 
the iterates $-q\omega \,\,(\mathrm{mod }\, 1)$ are $\gamma$-dense.


Since the frequency  $\omega$ is in $B_\gamma$, the equation 
$\|q\omega\|< \gamma X^{-1}$ has no solution for $|q|\leq X$ and any $X>0$. 
Therefore, Theorem \ref{thm:Cassels} implies that for any $\om \in \mathbb 
R/\mathbb Z$ there exists $q$ satisfying
\[
 \|q\omega-\alpha\|\leq \frac{1}{X} \quad 
\text{ and }\quad |q|\leq X\gamma^{-1}
\]
We take  $q^*=[X\gamma^{-1}]$.  Since we need $\gamma$-density, 
$X=\gamma^{-1}$. Then, using also \eqref{def:MaximalTime}, 
we obtain the following condition 
\[
\gamma^{-2}= |q^*| \geq\frac{1}{4C}\gamma\mu^{-\tau}-1\geq 
\frac{1}{5C}\gamma\mu^{-\tau}.
\]
Moreover, to apply Corollary \ref{corollary:KAM:eps-close}, one needs
\[
 \gamma\gtrsim \mu^{\frac{1-6\tau}{2}}.
\]
Thus, one can take, in particular  $\gamma\geq (5C)^{\frac 13} 
\mu^{\frac{1-6\tau}{2}}$. Then, it is easy to check that taking 
\[
 \tau=\frac{3}{20}
 ,
 \qquad \gamma=C\mu^{\frac{1}{20}}.
\]
for $C>1$ large enough independent of $\mu$, the two inequalities 
are satisfied. Moreover, this choice of $\gamma$,  optimizes 
the density of both KAM tori and the spreading of orbits in these tori.

\begin{rmk}\label{rmk:HermanKAM2}
If one considers regularity $C^{3+\beta}$ with $\beta>0$ small instead of $C^4$, as explained in Remark \ref{rmk:HermanKAM1}, one can proceed analogously. One would obtain then 
\[
 \tau=\frac{3}{17+3\beta}
 ,
 \qquad \gamma=C\mu^{\frac{1}{17+3\beta}}.
\]
This would lead to the improved density pointed out in Remark \ref{rmk:MainBetterDensity}.
\end{rmk}

\noindent   {\bf Two collisions in each torus:} \  The reasoning above has 
the simplifying assumption that each torus has only one collision 
instead of two. Now we incorporate the second collision. Note 
that the only problem of including the other collision is that 
the considered backward orbit departing from collision 1 
located at $(0,0)$ may have intersected the 
$4C\mu^\tau$-neighborhood of the other collision, where 
the two flows $\phi_{\widehat H_\mu}(t,z)$ and 
$\phi_{H_\mu}(t,z)$ differ, before reaching the final time 
$t=-q^*$. We prove that the backward orbit until time 
$-q^*$ from one collision may intersect 
the $4C\mu^\tau$-neighborhood of the other collision, 
but this \textit{cannot happen} for the $(-q^*)$-time 
backward orbits of the two collisions, just for one of them.

Assume that the collisions are located at $(0,0)$ and $(\ell',g')$. 
Call $(\ell'', 0)$ the first intersection between $g=0$ and the backward 
orbit  of the point $(\ell',g')$ under the flow \eqref{def:LinearFlow}
(see Figure \ref{actionangle}). 
The time to go  from $(\ell',g')$ to $(\ell'', 0)$ is independent of $\mu$ 
and, therefore, studying returns to the $1$-dimensional section suffices.
Assume that both the $(-q^*)$-backward orbit of $(0,0)$ hits a $4C\mu^\tau$ neighborhood of  $(\ell'',0)$ 
and the $(-q^*)$-backward orbit of $(\ell'',0)$ hits a $4C\mu^\tau$ neighborhood of  $(0,0)$. That is, 
there exist $0\leq q_1, q_2\leq q^*$ such that
\[
 \begin{split}
 \|q_1\omega-\ell'' \|&< 4C\mu^\tau\\
 \|q_2\omega+\ell'' \|&< 4C\mu^\tau.
 \end{split}
\]
Using the Diophantine condition
\[
 \frac{\gamma}{|q_1+q_2+2|}\leq  \|(q_1+q_2+2)\omega \|\leq  \|(q_1+1)\omega-\ell'' \|+ \|(q_2+1)\omega+\ell'' \|< 8C\mu^\tau.
\]
Therefore, $q_1+q_2>  8C\gamma\mu^{-\tau}-2$ which, by \eqref{def:MaximalTime}, implies that either $q_1$ or $q_2$ satisfies 
$q_i>  4C\gamma\mu^{-\tau}-1$. This contradicts $q_i\leq q^*$. 

Thus, the $(-q^*)$-backward orbit under the flow $\phi_{H_\mu}$ of 
one of the two collisions covers the torus $\mu^\frac{1}{20}$--densely.
Equivalently,  for any point $(\ell_0,g_0)$ in the torus, there exists a point 
$(\ell^*,g^*)$ which is $\mu^\frac{1}{20}$-close to a trajectory of the flow 
$\phi_{H_\mu}$ hitting either $\partial \mathcal B_{\mathcal T}^-$ or 
$\partial \mathcal B_{\mathcal T}^+$. Now, since the invariant tori are  
$\gamma\sim\mu^{\frac{1}{20}}$ dense in $\mathcal V_\de$ by 
Corollary \ref{corollary:KAM:eps-close}, we have that 
the $\mu^{\frac{1}{20}}$ neighborhood of any point in $\mathcal V_\de$ 
contains a point whose orbit reaches  either  
$\partial \mathcal B_{\mathcal T}^-$ or $\partial \mathcal B_{\mathcal T}^+$.

We do not want just one orbit to hit  $\partial \mathcal B_{\mathcal T}^\pm$  
but we want  a whole  segment of  length $\sim\mu^{\frac{3}{20}}$ to hit 
as stated in  Item 1 of Lemma \ref{lemma:DensityDelaunay}. 
Since we have considered coordinates such that the dynamics on $\cT$ is a rigid rotation, one can see that the 
orbit of any point $C\mu^\tau$-close to $(\ell^*,g^*)$ does not hit 
 $\mathcal B_{\mathcal T}^+$ for time $q^*+O(1)$ either.  Therefore, 
$\mu^{\frac{1}{20}}$-close to any point one can construct a segment 
which hits $\partial\mathcal B_{\mathcal T}^+$ as stated in Item 1 of 
Lemma \ref{lemma:DensityDelaunay}.

The considered coordinates are different but $\eps^{1/2}$-close to the original $(\ell,g)$ (recall that abusing notation we have kept the same notation for both systems of coordinates). Nevertheless, all the statements proven are coordinate free and, therefore, are still valid in the original $(\ell,g)$ coordinates.

Moreover, the localization in 
actions is a direct consequence from the graph property in Corollary 
\ref{corollary:KAM:eps-close}. 
Item 2 is a direct consequence of the fact that the constructed orbits do not intersect  $\mathcal B_{\mathcal T}^\pm$ until they hit its boundary at time $q^*+O(1)$. This completes the proof of Lemma \ref{lemma:DensityDelaunay}.

%

\subsection{Back to Cartesian coordinates: proof of Proposition \ref{prop:r1-region-main}}\label{sec:BackToCartesian}
To deduce Proposition \ref{prop:r1-region-main} from Lemma \ref{lemma:DensityDelaunay} it only remains to change coordinates to $(u,v)$. Note that the only statement which is not coordinate free in Lemma \ref{lemma:DensityDelaunay} is the graph property and localization in the variable $v$ in Item 1. To this end we need to analyze the change of coordinates $(\ell,g)\to u$ in a neighborhood of the collisions (note that the graph property  is only stated in these neighborhoods).

Using the Delaunay transformation and the graph property obtained in Lemma \ref{lemma:DensityDelaunay}, the segment $\mathcal{S}$ expressed in cartesian coordinates can be parameterized as 
\[
\begin{split}
u\equiv u(\ell, g, L,G)= u(\ell, g, L(\ell, g),G(\ell, g)), \ 
v\equiv v(\ell, g, L,G)= v(\ell, g, L(\ell, g),G(\ell, g)).
\end{split}
 \]
It only remains to show that we can invert the first row to express $(\ell,g)$ as a function of $u$. As a first step, we can express $(\ell, g)$ in terms of the polar coordinates $(r,\varphi)$. Using the definition of Delaunay coordinates, one can easily check that 
 \[
\left|\frac{\partial(\ell,g)}{\partial(r,\varphi)}\right|=\left|\begin{pmatrix}
   \partial_r \ell   &  0  \\
     \partial_r g  &  \partial_\varphi g 
\end{pmatrix}\right|= \left|\partial_r \ell \cdot \partial_\varphi g \right|=\left|\frac{1-e\cos u}{L^2e\sin u}\right|.
\]
The location of the collisions in Delaunay coordinates has been given in \eqref{def:CollisionDelaunay:mu}. This implies that in 
a $\mu^\tau$--neighborhood of the collisions
\[
 1-e\cos u=\frac{1}{L^2}+O(\mu^\tau)\neq 0
\]
Moreover, by condition \eqref{def:condLe}, $|\cos u|<1-\delta'$ for some $\delta'>0$ independent 
of $\mu$ and depending only on the parameter $\de$ introduced in \eqref{def:Vdelta}. This implies 
that $|\sin u|\geq \de''$ for some $\de''>0$ only depending on $\de'$. This implies that the change 
$(r,\varphi)\to (\ell,g )$ is well defined and a diffeomorphism in a $\mu^\tau$--neighborhood of the collisions. 
Since $(r,\varphi)\to u$ is a diffeomorphism, this gives the graph property stated in Proposition \ref{prop:r1-region-main}. 

Now, we need to prove the localization of the velocity $v$. To this end, it suffices to define the velocity $v_0$ as
\[
v_0=v\left(\ell_\col^{\pm,\mu},g_\col^{\pm,\mu},L_{\omega,\mu}(\ell_\col^{\pm,\mu},g_\col^{\pm,\mu}),G_{\omega,\mu}(\ell_\col^{\pm,\mu},g_\col^{\pm,\mu})\right).
\]
That is, the velocity $v$ evaluated on the (removed) collision point at the torus $\mathcal T$. Here the choice of $+$ or $-$ depends on the neighborhood of what collision the segment $\cS_0$ has hit. Using the smoothness of the torus, the estimate \eqref{drift} and estimates on the changes of coordinates just mentioned, one can obtain the localization in Item 1 of  Proposition \ref{prop:r1-region-main}.

Finally, let us mention that Lemma \ref{lemma:DensityDelaunay} considers $\cS_0\subset\partial \mathcal B_{\mathcal T}^\pm$ (see \eqref{def:PunctureInTorus}). On the contrary, Propostion  \ref{prop:r1-region-main} considers $\cS_0$ at $\partial \mathcal D_\pol^\pm$ (see \eqref{def:PuncturePolars}). These balls do not coincide since are expressed in different variables. Nevertheless, the boundaries are very close as stated in  \eqref{eq:LocalizationDpol}. Since the flow is close to integrable in the annulus  in \eqref{eq:LocalizationDpol}, one can  flow $\cS_0$ from $\partial \mathcal B_{\mathcal T}^\pm$ to $\partial \mathcal D_\pol^\pm$ keeping all the stated properties.
\end{proof}

\begin{figure}
\begin{center}
\includegraphics[width=8cm]{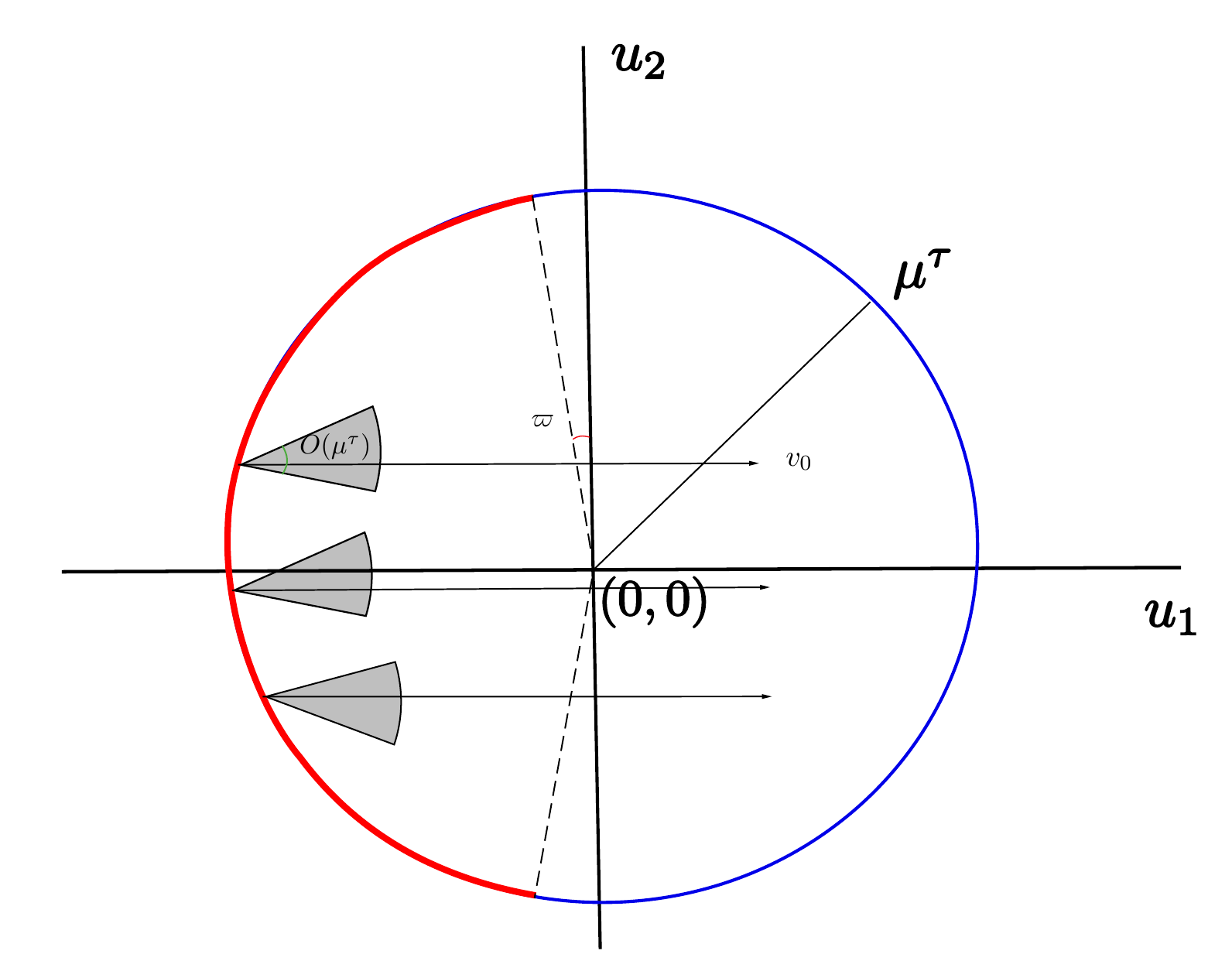}
\caption{Projection of $\cS_{0}$ onto the configuration space along with incoming velocity, which must belong to the grey cones.}
\label{pc}
\end{center}
\end{figure}


\section{The transition region $R_2$}\label{sec:transition}


In this section, we analyze the evolution of the segment $\mathcal S_{0}$ 
in the Transition Region (see \eqref{def:Regions}). More precisely, the goal 
is to prove that the evolution under the flow of $H_\mu$ of a subset of 
$\mathcal S_{0}$ reaches the inner boundary of the annulus $R_2$ 
(see \eqref{def:Regions}) and to obtain properties of this image set 
(see Fig. \ref{lr}).

To this end, we take $\rho>0$ and  we consider a section $\Gamma_1$ transversal to the flow
\begin{equation}\label{gamma1}
\begin{split}
\Gamma_1=&\Big\{\xi e^{i\pi/2}\frac{v_{0}}{|v_{0}|}\in\R^2\Big|\xi\in[-\mu^\tau,-\rho\mu^{1/2}\sec\frac{\varpi}{2}]\cup[\rho\mu^{1/2}\sec\frac{\varpi}{2},\mu^\tau]\ \Big\}\bigcup\nonumber\\
 &\Big\{\lb\rho\mu^{1/2}e^{i(\frac{\varpi}{2}+\frac{\pi}{2})}\frac{v_{0}}{|v_{0}|}+(1-\lb)\rho\mu^{1/2}\sec\frac{\varpi}{2}e^{i\pi/2}\frac{v_{0}}{|v_{0}|}\Big|\ \lb\in[0,1]\Big\}\ \bigcup\,\Gamma_{1,\varpi}
\end{split}
 \end{equation}
where
\begin{equation}\label{def:Gamma1'}
\Gamma_{1,\varpi}:=\left\{\rho\mu^{1/2} e^{i\theta}\cdot \frac{v_{0}}{|v_{0}|}\left|\theta\in\left[\frac\pi 2+\frac\varpi 2,\frac{3\pi}{2}-\frac \varpi2\right]\right.\right\},
\end{equation}
(see Fig. \ref{lr} ).
%
The main result of this section is the following.
\begin{prop}\label{prop:r2-region}
Consider the curve $\cS_{0}$ defined in Proposition \ref{prop:r1-region-main}. Then, for $\rho>0$ large enough and $\mu>0$ small enough,  there exists a subset $\cS'_{0}\subset\cS_{0}$  such that for all $P\in \cS'_{0}$ there exists 
a time $T_1(P)>0$ continuous in $P\in \cS'_{0}$ such that 
\[
\Gamma_{1,\varpi}\subset\pi_u\Big\{\phi_{H_\mu}(T_1(P),P):P\in\cS'_{0}\Big\}\subset \Gamma_1,
\]
where
$\phi_{H_\mu}(t,\cdot)$ is the flow associated to the Hamiltonian \eqref{r3b-e}. 

Moreover, if we denote by 
\[
\cS_{1}:=\Big\{\phi_{H_\mu}(T_1(P),P)\Big|\ P\in \cS'_{0}\Big\},
\]
 the following properties hold
\begin{itemize}
\item $\cS_{1}$ is a $C^0$ curve.
\item For all $P\in\cS_{1}$ , $\|\pi_vP-v_{0}\|\leq O(\mu^{\tau/3})$. 
\item For all $P\in\cS_{1}$, $T_1(P)\lesssim \mu^{\tau}$.
\end{itemize}
\end{prop}

To prove Proposition \ref{prop:r2-region} we first consider a first order of 
the equations associated to Hamiltonian $H_\mu$ in \eqref{r3b-e}. Taking 
into account that in the region $R_2$ we have that $|u|\leq\mu^\tau$ 
(see \eqref{def:Regions}), we define the Hamiltonian 
\begin{equation}\label{def:HamLinear}
H_{\lin}(u,v)=\frac{|v|^2}{2}-u^tJ v,
\end{equation}
which will be a ``good first order'' of $H_\mu$ and whose equations are linear,
\[
\begin{split}
  \dot u_1&=v_1+u_2\\
  \dot u_2&=v_2-u_1\\
  \dot v_1&=v_2\\
  \dot v_2&=-v_1.
  \end{split}
\]
  \begin{lem}\label{lemma:R2:linear}
Consider the curve $\cS_{0}$ defined in Proposition \ref{prop:r1-region-main}. Then,  there exists a subset $\cS^\lin_{0}\subset\cS_{0}$  such that for all $ P\in \cS^\lin_{0}$ there exists a time $T_\lin(P)>0$ continuous in $P\in \cS^\lin_{0}$ such that 
\begin{equation}\label{covering-gamma-1}
\Gamma_{1,\varpi}\subset\pi_u\Big\{\phi_{H_\lin}({T_\lin(P)},P):P\in\cS^\lin_{0}\Big\}\subset \Gamma_1,
\end{equation}
where $\Gamma_{1,\varpi}$ has been defined in \eqref{def:Gamma1'} and $\phi_{H_\lin}(t,\cdot)$ is the flow associated to Hamiltonian (\ref{def:HamLinear}). Moreover, if we define
\[
\cS^\lin_{1}=\Big\{\phi_{H_\lin}({T_\lin(P)},P)\Big|P\in \cS^\lin_{0}\Big\},
\]
 the following properties hold
\begin{itemize}
\item $\cS^\lin_{1}$ is a $C^0$ curve.
\item For all $P\in\cS^\lin_{1}$ , $\|\pi_vP-v_{0}\|\leq O(\mu^{\tau/3})$. 
\item For all $P\in\cS^\lin_{1}$, $T_\lin(P)\leq O(\mu^{\tau})$.
\end{itemize}
\end{lem}
\begin{proof}
 The proof of this lemma is straightforward taking into account that $|u|\lesssim\mu^\tau$ in $R_2$, that the trajectories associated to the Hamiltonian in \eqref{def:HamLinear} are explicit and given by 
 \[
\begin{split}
 \begin{pmatrix} u_1\\u_2\end{pmatrix}&=\begin{pmatrix}\cos t&\sin t\\-\sin t&\cos t\end{pmatrix} \begin{pmatrix} u_1^0\\u^0_2\end{pmatrix}+t\begin{pmatrix}\cos t&\sin t\\-\sin t&\cos t\end{pmatrix} \begin{pmatrix} v_1^0\\v^0_2\end{pmatrix}\\
 \begin{pmatrix} v_1\\v_2\end{pmatrix}&=\begin{pmatrix}\cos t&\sin t\\-\sin t&\cos t\end{pmatrix} \begin{pmatrix} v_1^0\\v^0_2\end{pmatrix}
 \end{split}
 \]
and the fact that $(v_1^0,v^0_2)$ has a lower bound independent of $\mu$.
\end{proof}

\begin{figure}
\begin{center}
\includegraphics[width=10cm]{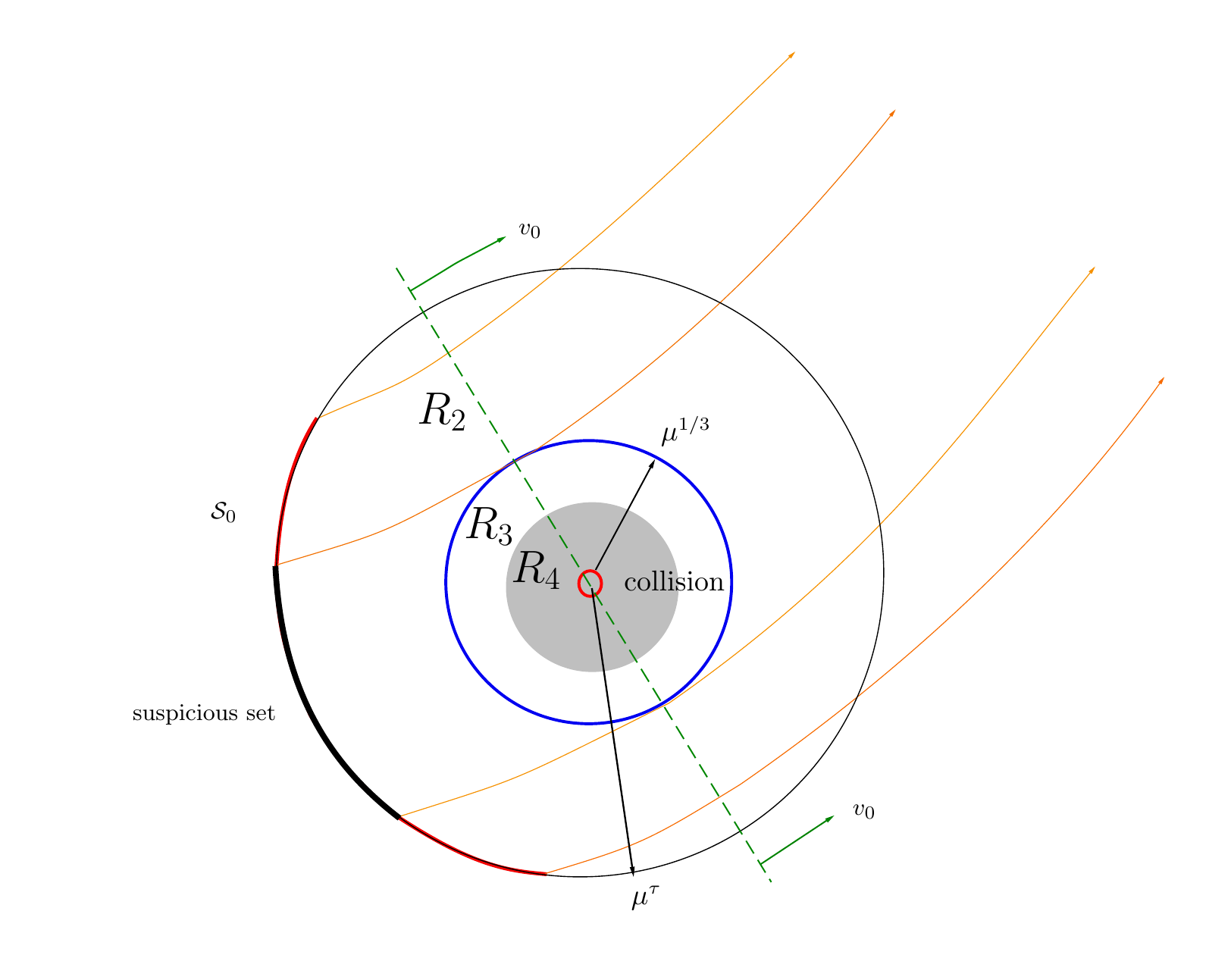}
\caption{Geometry of the incoming curve near collisions, see
 \eqref{def:Regions}.}
\label{pc}
\end{center}
\end{figure}


Once Lemma \ref{lemma:R2:linear} has given the behavior in Region $R_2$ 
of the flow associated to the Hamiltonian \eqref{def:HamLinear}, now 
we compare its dynamics to those of $H_\mu$ in \eqref{r3b-e}.

\begin{lem}\label{flow-est}
Take $\rho>0$ large enough and $\mu>0$ small enough.
Then, for all $P\in \cS_0$,  there exists  $T_1(P)>0$ continuous in $P$ satisfying
\begin{equation}\label{real-time}
|T_1(P)-T_\lin(P)|\lesssim \rho^{-1}\mu^{2\tau}
\end{equation}
such that  $\pi_u\phi_{H_\mu}(t,P)\in\mathrm{Int}(R_2)$ for all $t\in (0, T_1(P))$, $\pi_u\phi_{H_\mu}(T_1(P),P)\in\Gamma^1$ with 
\begin{equation}\label{real-v}
\left\|\pi_v\phi_{H_\mu}(T_1(P),P)-\pi_v\,\phi_{ H_\lin}(T_\lin(P),P)\right\|\lesssim\rho^{-1}\mu^{2\tau}.
\end{equation}
\end{lem}
\begin{proof}

The region $R_2$ satisfies $|u|\leq \mu^\tau$. Therefore, the equation associated to Hamiltonian $H_\mu$ in \eqref{r3b-e} satisfies
\[
\begin{split}
  \dot u_1&=v_1+u_2\\
  \dot u_2&=v_2-u_1\\
  \dot v_1&=v_2+O\left(\rho^{-1}+\mu\right)\\
  \dot v_2&=-v_1+O\left(\rho^{-1}+\mu\right).
  \end{split}
\]
Since $\rho$ is taken such that $\rho^{-1}\gg \mu$; we have that this equation is  $O(\rho^{-1})$--close to the equation of $H_\lin$ (see \eqref{def:HamLinear}). 

Consider the trajectory $(u(t), v(t))$ of  $(u^0,v^0)\in \cS_0$ under the flow of $H_\mu$. Then, applying variation of constants formula, as long as the trajectory remains in $R_2$, we have
 \[
\begin{split}
 \begin{pmatrix} u_1\\u_2\end{pmatrix}&=\begin{pmatrix}\cos t&\sin t\\-\sin t&\cos t\end{pmatrix} \begin{pmatrix} u_1^0\\u^0_2\end{pmatrix}+t\begin{pmatrix}\cos t&\sin t\\-\sin t&\cos t\end{pmatrix} \begin{pmatrix} v_1^0\\v^0_2\end{pmatrix}+O\left(\rho^{-1} t^2\right)\\
 \begin{pmatrix} v_1\\v_2\end{pmatrix}&=\begin{pmatrix}\cos t&\sin t\\-\sin t&\cos t\end{pmatrix} \begin{pmatrix} v_1^0\\v^0_2\end{pmatrix}+O\left(\rho^{-1} t\right)
 \end{split}
 \]
Then, it is straightforward to prove \eqref{real-time} and \eqref{real-v}.
\end{proof}
Recall that for any starting point $(u^0,v^0)\in \cS_0$, 
we know $\|v^0-v_0\|\lesssim\mu^{\tau/3}$. From Lemmas \ref{lemma:R2:linear} and \ref{flow-est}, one can easily 
deduce the proof of Proposition \ref{prop:r2-region}.

\section{Levi-Civita Regularization in the region $R_3$}\label{sec:LeviCivita}

The last step to prove Theorem \ref{main} is to show that 
there is a point inside the curve $\cS_1$ (from Proposition \ref{prop:r2-region}) whose trajectory hits a collision. 
To this end we analyze a $\rho\mu^{1/2}$--neighborhood of 
the collision $u=0$ by means of the Levi-Civita regularitzation.

For $|u|\leq\rho\mu^{1/2}$, system $H_\mu(u,v)$ can be
expanded as
\begin{equation}\label{def:HamExpR4}
H_\mu(u,v)=\frac{|v|^2}{2}-u^tJ v-\frac{\mu}{|u|}-\frac12(\mu-1)(\mu-3)-\frac12(1-\mu)(2u_1^2-u_2^2)+O(u^3).
\end{equation}
Performing the following scaling and time reparamaterization
\begin{equation}\label{def:scaling}
u=\mu^{1/2} \wt u, \quad t=\mu^{1/2}\varsigma,
\end{equation}
we obtain a new system, which is Hamiltonian with respect to 

\begin{equation}\label{def:Hrho}
\widetilde{H}_\rho(\wt u,v)=\frac
12|v|^2-\mu^{1/2}\wt u^tJv-\mu^{1/2}\frac{1}{|\wt u|}
-\frac12 \mu(1-\mu)(2\wt u_1^2-\wt u_2^2)+O(\mu^{3/2} \wt u^3).
\end{equation}
Recall that we have fixed
$C_J\in(-2\sqrt2,3)$. Thus, for $\mu\ll1$ small enough, the
energy of $\widetilde H_\rho(\wt u,v)$ belongs to $(0,\sqrt2+3/2)$ (note the
constant term $(\mu-1)(\mu-3)/2$ in \eqref{def:HamExpR4} and recall that $C_J=-2
H_\mu$).

Consider the set $\Gamma_{1,\varpi}$ introduced in \eqref{def:Gamma1'}. We express it in the new coordinates
 \begin{equation}\label{def:Gamma2:Scaled}
\Gamma^0_{1,\varpi}=\left\{\rho e^{is}\cdot \frac{v_{0}}{|v_{0}|}\left|\ s\in\left[\frac\pi 2+\frac\varpi 2,\frac{3\pi}{2}-\frac \varpi2\right]\right.\right\}.
 \end{equation}
We want to apply the Levi Civita regularization to the Hamiltonian
$\widetilde{H}_\rho(\wt u,v)$ restricted to fixed level of energies. 
To this end, we introduce the constant $\xi$ which 
represents the energy of $\widetilde H_\rho$ as $\widetilde
H_\rho(\wt u,v)=\frac{1}{2\xi^2}$. Denote by 
$\widetilde H_\rho^0(\wt u,v)$, the Hamiltonian containing 
the ``leading'' terms of $\widetilde{H}_\rho$, 
\[\widetilde H_\rr^0(\wt u,v)=\frac 12|v|^2-\mu^{1/2}\wt
u^tJv-\mu^{1/2}\frac{1}{|v|}.\]
Then the difference between $\widetilde H_\rr^0(\wt u,v)$ and $\widetilde
H_\rho(\wt u,v)$ 
satisfies $\|\widetilde H_\rho(\wt u,v)-\widetilde
H_\rr^0(\wt u,v)\|_{C^3}\leq O(\mu)$. \\

Fix $\varpi>0$  a small constant independent of 
$\mu$ and $\rr$ and a level of energy in $(0,\sqrt2+3/2).$
The goal of this section is to study which orbits starting at 
$\wt u=\rho e^{is}$, with  
$s\in[\frac\pi 2+\varpi,\frac{3\pi}{2}-\varpi]$,
tend to collision. We analyze them by considering 
the Levi-Civita transformation
\begin{equation}\label{lc-tran}
\wt u=2z^2,\quad v=\frac{w}{\xi\bar z}
\end{equation}
with $\wt u\in\R^2\cong \C$ uniquely identified by a  complex number and
$\xi\sim
O(1)$ being a scaling constant depending on the energy. Applying this
change of coordinates and a time scaling to $\widetilde H_\rho$ in \eqref{def:Hrho}, we obtain a
new system which is Hamiltonian with respect to
\[K_\rho(z,w)=\xi^2|z|^2\Big[\widetilde H_\rho\left(2z^2,\frac{w}{\xi\bar
z}\right)-\frac{1}{2\xi^2}\Big].\]
Note that the change of time is regular only away from collision $z=0$. At $z=0$ it regularizes the collisions. 

The change of coordinates \eqref{lc-tran} implies that $K^{-1}_\rho(0)\backslash\{z=0\}$ defines a two-fold covering of
the energy  surface $\widetilde H_\rho^{-1}({1}/{2\xi^2})\setminus\{u=0\}$. Moreover, the flow
on $K_\rho^{-1}(0)\backslash\{z=0\}$ becomes the flow on $\widetilde
H_\rho^{-1}({1}/{2\xi^2})\setminus\{u=0\}$ via the time reparametri\-zation.

In the new coordinates $(z,w)$,  the section $\Gamma^0_{1,\varpi}$ in \eqref{def:Gamma2:Scaled}
becomes 
 \[
 \widetilde\Gamma_{1,\varpi}^0=\left\{z=\sqrt{\frac{\rho}{2}}\,
 e^{i(s+s_0)}\left| \ s\in \left[\frac\pi 4+\frac\varpi 4,\frac{3\pi}{4}-\frac \varpi4\right]
 \right. \right\} 
 \]
where $2s_0$ is the argument of $v_{0,w}$. Define $\displaystyle
\widetilde\rr=\sqrt{\frac{\rho}{2}}$. 

If one restricts $ \widetilde\Gamma_{1,\varpi}^0$ to the zero
level of energy, that is $\widetilde\Gamma_2^0\cap K^{-1}_\rho(0)$, one has 
$|z|=\widetilde\rho$ and $|w|=\widetilde\rho+O(\mu^{1/2})$. Thus, since
$ \widetilde\Gamma_{1,\varpi}^0\cap K^{-1}_\rho(0)$ is two dimensional, it can be
parameterized by the arguments of $z$ and $w$. We can express $K_\rho(z,w)$ as 
\begin{equation}\label{lc-eq}
\begin{split}
K_\rho(z,w)=&\frac
12(|w|^2-|z|^2)-\frac12\mu^{1/2}\xi^2\\
&
-2i\xi^2\mu^{1/2}|z|^2(\bar z w-\bar w
z)
-\frac12(1-\mu)\mu\xi^2\Big[2|z|^6+3|z|^2(z^4+\bar z^4)+O(z^8)\Big],
\end{split}
\end{equation}
with $(z,w)\in B(0,O(\widetilde\rho))\subset \C^2$. 
Taking into account that $|z|=\widetilde\rho$ and 
$|w|\sim \widetilde\rho$, the second line is of higher order
compared to the first one. 

We want to analyze the orbits which hit a collision. In  coordinates $(z,w)$, 
this corresponds to orbits intersecting $\{z=0\}$. Equivalently, we analyze 
orbits with initial condition at $\{z=0\}$ at the energy surface 
$K_\rho^{-1}(0)$ and we consider their backward trajectory.

Consider the first order of the Hamiltonian  \eqref{lc-eq}, given by 
\begin{equation}\label{def:HamLC:quadratic}
 K^0_\rho(z,w)=\frac 12(|w|^2-|z|^2)-\frac12\mu^{1/2}\xi^2.
\end{equation}
It  has a resonant saddle
critical point $(0,0)$, with $1$ as a positive eigenvalue of multiplicity two.  
We analyze the dynamics of the quadratic Hamiltonian at the energy 
surface $K_\rho^{-1}(0)$. Later we deduce that the full system has 
approximately the same behavior. 

We consider collisions points at $K_\rho^{-1}(0)$ as initial condition. That is, by \eqref{def:HamLC:quadratic}, points of the form
\begin{equation}\label{def:InitialCondLeviCivita}
 z=0,\quad w=\delta_\mu e^{i\psi}\quad\text{with}\quad
\delta_\mu=\mu^{1/4}\xi\quad\text{ and }\quad \psi\in\mathbb R/(2\pi\mathbb Z).
\end{equation}
Consider
an initial condition of the form
\eqref{def:InitialCondLeviCivita} and call $(z(t),w(t))$ the corresponding
orbit under the flow of \eqref{def:HamLC:quadratic}
\begin{lem}
Fix $\varpi>0$ small and  a closed interval 
$I\subset (0, 2\sqrt{2}+3)$. Then for $\mu$ small
enough and
$\xi$ with $1/(2\xi^{2})\in I$, after time 
\[ T=-\mathrm{arcsinh}\left(\frac{\widetilde\rr}{\de_\mu}\right)=-\log
\frac{2\widetilde\rr}{\de_\mu}+O(\de_\mu^2)<0\]
the orbit satisfies $(z(T),w(T))\in  \widetilde\Gamma_{1,\varpi}^0$
 and the image contains the curve
\begin{equation}\label{eq:LC:Linear}
\left\{(w,z)\in \widetilde\Gamma_{1,\varpi}^0: \quad 
\mathrm{arg}(w)=\mathrm{arg}(z)-\pi+O(\mu^{1/4}), \quad \mathrm{arg}(z)\in
\left[\frac\pi
2+\varpi,\frac{3\pi}{2}-\varpi\right]\right\}.
\end{equation}
\end{lem}

\begin{proof}The proof of this lemma is a direct consequence of the integration of the
linear system associated to Hamiltonian \eqref{def:HamLC:quadratic}. Indeed, the trajectory associated to this system with initial condition \eqref{def:InitialCondLeviCivita} is given by
\[
 \begin{split}
  z(t)&=\delta_\mu e^{i\psi}\sinh t\\
  w(t)&=\delta_\mu e^{i\psi}\cosh t
 \end{split}
\]
Thus taking $T<0$ as stated in the lemma the orbits reach $\widetilde\Gamma_{1,\varpi}^0$ and satisfy \eqref{eq:LC:Linear}
\end{proof}

The next lemma shows that if one considers the full Hamiltonian \eqref{lc-eq}, 
the same is true with a small error.  Call $(z(t), w(t))$ to the orbit with initial condition of the form  \eqref{def:InitialCondLeviCivita} under the flow associated to  \eqref{lc-eq}.
\begin{lem}\label{collisional-cover}
Fix $\varpi>0$  small,  a closed interval $I\subset (0, 2\sqrt{2}+3)$ and 
an initial condition of the form \eqref{def:InitialCondLeviCivita}.  
Then, for $\mu$ small enough and $\xi$ with 
$1/(2\xi^{2})\in I$, there exists a time $T<0$ (depending on the initial 
condition), satisfying
 \[\left|T+\log \frac{2\widetilde\rr}{\de_\mu}\right|\leq C\mu^{1/4}\]
for some $C>0$ independent of $\mu$, such that 
$(z(T), w(T))\in \widetilde\Gamma_{1,\varpi}^0$.

Moreover, the intersection between  $ \widetilde\Gamma_{1,\varpi}^0$ and  the union of
orbits with  initial conditions $\eqref{def:InitialCondLeviCivita}$ with any
$\psi\in [0,2\pi]$  contains a continuous curve 
$(z,w)=(\gamma_1(\psi), \gamma_2(\psi))$ which satisfies
\[\mathrm{arg }\, z(\psi_1)=\frac{\pi}{ 2}+\varpi\qquad  , \qquad 
\mathrm{arg }\,z(\psi_2)=\frac{3\pi}{2}-\varpi\]
for some $\psi_1<\psi_2$, $\psi_1,\psi_2\in  [0,2\pi]$, and 
\[ \left|
\mathrm{arg }\gamma_1(\psi)-
\mathrm{arg }\gamma_2(\psi)-\pi\right|\leq O\left(\mu^{1/4}\right).\]
\end{lem}

\begin{proof}
We prove the lemma by using the variation of constants formula. Consider the
symplectic change of coordinates
\begin{equation}\label{def:DiagLC}
 X_i=\frac{z_i+w_i}{\sqrt{2}},\qquad  Y_i=\frac{z_i-w_i}{\sqrt{2}},\qquad i=1,2,
\end{equation}
which transforms $K_\rho^0$ into 
\[
 \widetilde
K_\rho^0=\frac12\left(X_1Y_1+X_2Y_2\right)-\frac12\mu^{1/2}\xi^2+\mu^{1/2}
O_4(X,Y).
\]
We consider the corresponding initial condition 
$X_0=\frac{\de_\mu e^{i\psi}}{\sqrt{2}}$, 
$Y_0=\frac{-\de_\mu e^{i\psi}}{\sqrt{2}}$ and 
the equations associated to $\widetilde K_\rho^0$, which are of the form
\[
\begin{split}
\dot X&=\ X+\mu^{1/2}O_3(X,Y)\\ 
\dot Y&=-Y+\mu^{1/2}O_3(X,Y).
\end{split}
\]
We obtain estimates by using a bootstrap argument. Call $T^*<0$ the first time
such that $(X(t),Y(t))$ leave the ball of radius one (if it does not exist, set 
$T^*=-\infty$). Then, using the variation of constants formula, we have that
for $t\in (T^*,0)$,
\[
\begin{split}
 X(t)=\ e^{t}\left(X_0+O(\mu^{1/2})\right)\\
 Y(t)\ =\ e^{-t}Y_0+O(\mu^{1/2}).
  \end{split}
\]
Using the value of $X_0$ and $Y_0$,  there exists $T<0$ depending on
$(X_0,Y_0)$ satisfying that 
\[\left|T+\log \frac{2\widetilde\rr}{\de_\mu}\right|\leq C\mu^{1/2}\]
for some $C>0$
independent of $\mu$ (but depending on $\rr$) such that  the corresponding
$(z(T), w(T))$ (by \eqref{def:DiagLC}) belongs to
$\widetilde\Gamma_2^0$ and satisfy
\[
\mathrm{arg}\,z(T)=\psi+\pi+O\left(\mu^{1/4}\right),\qquad
\mathrm{arg}\,w(T)=\psi+O\left(\mu^{1/4}\right).
\]
This implies the statements of the lemma.
\end{proof}

Undoing the changes of coordinates \eqref{def:scaling} and \eqref{lc-tran}, 
we can analyze the orbits leading to collision for the Hamiltonian \eqref{r3b-e}.

\begin{cor}\label{coro:cross2}
For $\varpi>0$ small there exists a curve 
$\Upsilon=\{(u,v)=(u(\psi),v(\psi))\subset\mathbb R^4: \psi\in J\}$ where $J \subset\mathbb R$ is an interval such that: 
\begin{enumerate}
\item The projection of 
$\Upsilon$ onto the $u$ variable contains the set
\[
\Gamma'_{1,\varpi}=\left\{\rho \mu^{1/2} e^{i\theta}\cdot\frac{v_{0}}{|v_{0}|}
\left|\  \theta\in
 \left[\frac{\pi}{2}+\varpi,\frac{3\pi}{2}-\varpi\right]\right.\right\},
\] 
\item It satisfies
\[
\begin{split}
u(\psi)&=\rho \mu^{1/2} e^{2i\psi}\left(1+O\left(\mu^{1/4}\right)\right)\\
v(\psi)&=-\xi^{-1} e^{2i\psi}\left(1+O\left(\mu^{1/4}\right)\right)
\end{split}
\]
\item The orbits of  the  Hamiltonian $H_\mu$ in \eqref{r3b-e} with initial
condition in $\Upsilon$ hit a collision.
\end{enumerate}
\end{cor}

\begin{figure}
\begin{center}
\includegraphics[width=10cm]{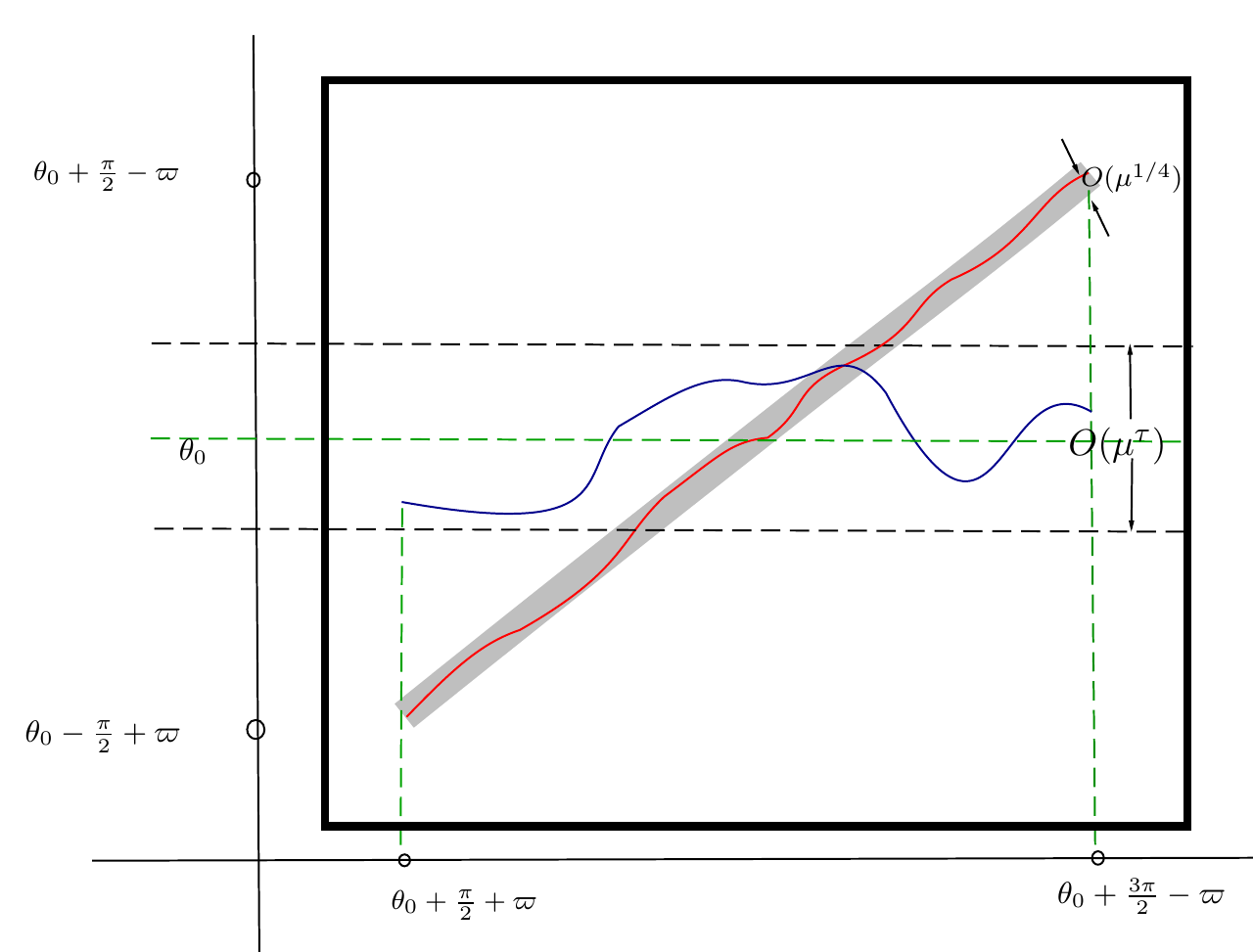}
\caption{The Blue curve is the projection of $\mathcal S_1$ obtained in Proposition \ref{prop:r2-region} onto the $(\mathrm{arg}(u), \mathrm{arg}(v))$ plane whereas the red curve is the projection onto the same plane of the curve $\Upsilon$ obtained in Corollary \ref{coro:cross2}. We use the notation $\theta_0=\mathrm{arg}(v_0)$.
}
\label{cross}
\end{center}
\end{figure}

Proposition \ref{prop:r2-region} and Corollary \ref{coro:cross2} imply 
Theorem \ref{main}. Indeed, it only remains to prove that the segment $\cS_1$ 
obtained in Proposition \ref{prop:r2-region} and the segment $\Upsilon$ 
obtained in Corollary \ref{coro:cross2} intersect. Note that both curves project 
onto $\Gamma_{1,\varpi}$ in \eqref{covering-gamma-1} and belong to 
the same level of energy of the Hamiltonian $H_\mu$ in \eqref{r3b-e}. 
Therefore, these two curves belong to the two dimensional surface
\[
 \mathcal M_h=\left\{(u,v)\in\mathbb R^4: \ 
|u|=\rho\mu^{1/2},\quad H_\mu(u,v)=h\right\}
\]
for some $h\in\mathbb R$. Therefore, to complete the proof of Theorem \ref{main}, we only need to prove that the two curves intersect in this 2 dimensional surface. To parameterize $\mathcal M_h$, taking into account that 
$|u|=\rho\mu^{1/2}$ and that this implies
\[
h=H_\mu(u,v)=\frac{|v|^2}{2}+O\left(\mu^{1/2}\right),
\]
one can consider as  variables the arguments of $u$ and $v$. In these coordinates, the two continuous curves $\cS_1$ and $\Upsilon$ satisfy the following:
\begin{itemize}
 \item By Proposition \ref{prop:r2-region}, the projection onto the argument of $u$ of the curve $\cS_1$ contains the interval \[\displaystyle\left[\mathrm{arg}(v_0)+\frac{\pi}{2}+\frac{\varpi}{2},\mathrm{arg}(v_0)+\frac{3\pi}{2}-\frac{\varpi}{2}\right]\] whereas the $v$ component satisfies $\mathrm{arg}(v)=\mathrm{arg}(v_0)+O(\mu^\tau)$. That is, in the plane $(\mathrm{arg}(u),\mathrm{arg}(v))$ is a  curve close to horizontal.
 \item By Corollary \ref{coro:cross2}, the projection onto the argument of $u$ of the curve $\Upsilon$ contains $\displaystyle\left[\mathrm{arg}(v_0)+\frac{\pi}{2}+\frac{\varpi}{2},\mathrm{arg}(v_0)+\frac{3\pi}{2}-\frac{\varpi}{2}\right]$. Moreover, $\Upsilon$ satisfies \[\mathrm{arg}(v)=\mathrm{arg}(u)-\pi+O(\mu^{1/4}).\]
\end{itemize}
Since the two curves are continuous, they must intersect. This completes the proof of Theorem \ref{main}.

\section{Proof of Theorem \ref{main1}}\label{sec:Wandering}
To prove Theorem \ref{main1} we use the ideas developed in Section \ref{sec:R1} to analyze the region $R_1$. We only need to modify the density argument from the one given in Section \ref{sec:DensityDelaunay}. As explained in Section \ref{sec:BackToCartesian},  the change from Delaunay 
to the Cartesian coordinates $(u,v)$ is a diffeomorphism with uniform bounds 
independent of $\mu$. Therefore, it is enough to prove Theorem \ref{main1} in Delaunay coordinates.

Theorem \ref{main1} is a consequence of the following lemma. We use the notation of Section \ref{sec:R1}: we consider the tori $\mathcal{T}$ given by Corollary \ref{corollary:KAM:eps-close} and we denote by $B_{\mathcal{T}}^\pm$ the balls of radius $C\mu^\tau$ in these tori centered at collisions (see \eqref{def:PunctureInTorus}). The Hamiltonians $H_\mu$ in \eqref{def:HamPolars} (expressed in Delaunay coordinates) and $\widehat H_\mu$  in \eqref{i-i-h}
coincide away from $B_{\mathcal{T}}^\pm$.
\begin{lem}\label{lemma:DensityDelaunayrecurrent}
Fix $\de>0$ small, there exists $\mu_0>0$ 
depending on $\de$, such that the following holds 
for any $\mu\in (0,\mu_0)$:

For any ${\mathbb  X}\in\cV_\de$, there exists a invariant torus $\cT$ 
obtained in Corollary \ref{corollary:KAM:eps-close} and 
a point $\mathbb Y\in \cT$ satisfying  
$\mathrm{dist}(\mathbb Y, \mathbb X)\leq O(\mu^{\frac{1}{20}})$, such that 

\begin{enumerate}
\item {\sf (Away from the collision)} There exists $0<T(\mathbb Y)\lesssim O(\mu^{-\frac{1}{10}})$, such that for all $t\in (0, T(\mathbb Y))$, 
$\phi_{\widehat  H_\mu}(t,z)\not\in 
\mathcal B_{\mathcal T}^+\cup \mathcal B_{\mathcal T}^-$;
Therefore, we have 
\[
\phi_{\widehat H_\mu}(t,\mathbb Y)=
\phi_{H_\mu}(t,\mathbb Y),\quad \text{ for all }\,
 t\in[0,T(\mathbb Y)].
\]

\item {\sf (Recurrence)} $\mathrm{dist}( \phi_{H_\mu}(T(\mathbb Y), \mathbb Y), \mathbb X)\leq O(\mu^{\frac{1}{20}})$.

\item {\sf (Close to collision)} There exists $T'(\mathbb Y)\in (0,T(\mathbb Y)]$, such that 
\[
\mathrm{dist}(\mathbb \phi_{H_\mu}(T'(\mathbb Y), \mathbb Y), \cB_\cT^\pm)\leq O(\mu^{\frac{1}{20}}).
\]
 \end{enumerate}
\end{lem}

We devote the rest of the section to prove this lemma. The reasoning follows the same lines as that of Section \ref{sec:DensityDelaunay}. Namely, since the point $\mathbb Y$ considered is contained in one of the KAM tori $\cT$ from  
Corollary \ref{corollary:KAM:eps-close} we need to optimize $\ga$ (see \eqref{eq:dioph-gamma}) so that we get enough density of tori in Corollary \ref{corollary:KAM:eps-close} and strong enough Diophantine condition so that the orbits of $\widehat H_\mu$ are well spread in $\cT$.
\subsection{Proof  of Lemma \ref{lemma:DensityDelaunayrecurrent}}
Fix $\mathbb X\in\cV_{\delta}$ and consider a torus  $\mathcal T$  among the ones given in Corollary \ref{corollary:KAM:eps-close} $\ga$-close ot it with $\ga$ to be determined. We look for a point $\mathbb Y$ in this torus satisfying the statements of Lemma \ref{lemma:DensityDelaunayrecurrent}. To this end, we look for an orbit in $\cT$ spreading densely enough on the torus.

We proceed as in Section \ref{sec:DensityDelaunay}. Corollary \ref{corollary:KAM:eps-close} implies that  $\cT$ is a graph over $(\ell, g)$ and the dynamics on $\cT$ is $\eps^{1/2}=\mu^{(1-6\tau)/2}$-close to the unperturbed one. Moreover, after a $\eps^{1/2}$-close to the identity transformation, the dynamics (projected to the base) is a rigid rotation, which by a time reparamaterization, is given by
\[
(\ell(t),g(t))=(\ell^0+\omega t, g^0+t)
\]
where $\omega\in B_{\ga}$ (see \eqref{eq:dioph-gamma}).

It is enough to analyze the orbits in $\cT$ in these coordinates. We analyze the density of orbits in $\cT$ on the section $\{g=0\}$. Since 
the dynamics is a rigid rotation, the density in the section implies 
the density in the whole torus. 

Proceeding as in Section \ref{sec:DensityDelaunay}, we first assume that each torus has just one collision  and then we adapt the proof to deal with tori having two collisions.\\

\noindent {\bf One collision model case:} \  Consider the point $z_0$ on the same horizontal as the collision $\mathcal C_+$ with $\ell$ coordinate $4C\mu^\tau$ bigger. This point is outside of the puncture $B^+_\cT$ since it has radius $C\mu^\tau$ (see \eqref{def:PunctureInTorus}). By a translation we can assume that $z_0=(0,0)$ and the collision is at $\mathcal C_+=(-4C\mu^\tau, 0)$.

In Section \ref{sec:DensityDelaunay} we have considered the backward orbit of $(0,0)$. Since now we want a non-wandering result, we consider both the forward and backward orbits. We want both of them to cover $\ga$-densely the torus without intersecting the $B^+_\cT$. As explained in Section \ref{sec:DensityDelaunay}, it is enough to consider the intersections of the orbit with $\{g=0\}$ given by $\|q\omega\|$ (see \eqref{def:norm}) for $q=-q^*,\ldots,q^*$ with 
\begin{equation}\label{def:MaxIteratesRec}
 q^*=\left\lceil\frac{1}{20C}\ga\mu^{-\tau}-1\right\rceil.
\end{equation}
The Diophantine condition \eqref{eq:dioph-gamma} implies that 
$\|q\omega\|\geq 20C\mu^\tau$ for $q=-q^*,\ldots,q^*$ and, therefore, none of these iterates belong to $B_\cT^+$. 
Moreover, applying Theorem \ref{thm:Cassels} and  choosing 
\[
 \tau=\frac{3}{20} \qquad \text{ and }\qquad \gamma\sim \mu^{\frac{1}{20}},
\]
one can see (as in Section \ref{sec:DensityDelaunay}) that both the forward and the backward orbits are $O(\ga)$-dense in the torus.

If the torus $\cT$ would have only one collision, this would complete the proof of Lemma \ref{lemma:DensityDelaunayrecurrent}. Indeed, the $O(\ga)$-neighborhood of any point in $\cT$ intersects both the forward and the backward orbit of $z_0$. Since the tori are $\ga$-dense (Corollary \ref{corollary:KAM:eps-close}), for any point $\mathbb X\in\cV_\delta$, there exists a torus $\cT$ $\ga$-close to it and a point $\bY$  which belongs to the just constructed backward orbit on this torus $\cT$ which is also $O(\ga)$-close to $\mathbb X$. If one considers now the forward orbit of $\bY$, after  time $T\sim \ga\mu^{-\tau}\sim\mu^{-1/10}$ there is an iterate of the orbit which is $O(\ga)$-close to $\bY$ and therefore $O(\ga)$-close to $\mathbb X$. Moreover, this orbit has not intersected $B_\cT$.\\

\noindent  {\bf Two collisions case:} \ Now we show that the same reasoning goes through if we include the second collision of the torus. If we add the second collision, there are two possibilities:
\begin{itemize}
\item If the orbit of $z_0$ does not intersect $B_\cT^-$ for the considered times the proof of Lemma \ref{lemma:DensityDelaunayrecurrent} is complete. 
\item If the orbit of $z_0$ does intersect  $B_\cT^-$, we move slightly $z_0$ to have an orbit with the same properties as the previous one and not intersecting either of $B_\cT^\pm$.
\end{itemize}
We devote the rest of the section to deal with the second possibility. We use the same system of coordinates as before, which locates $z_0=(0,0)$ and the first collision at $\mathcal C_+=(-4C\mu^\tau, 0)$. We denote the second collision by $\mathcal C_-=(\ell',g')$. Call $\mathcal C_-'=(\ell'',0)$ the first intersection between $\{g=0\}$ and the backward orbit of $\mathcal C_-$. Since the time to go from one point to the other is independent of $\mu$, it is enough to study the forward and backward orbit of $z_0$ in the section $\{g=0\}$.

By assumption, there exists $q'$ with $|q'|\leq q^*$ such that 
\begin{equation}\label{def:Recu:CollisionHitting}
\|q'\omega-\ell''\|\leq 4C\mu^\tau.
\end{equation}
Then, we consider a new point $z_1=(\ell_1,0)=(10C\mu^\tau,0)$, which is $10C\mu^\tau$ far away from  $z_0$ and $14C\mu^\tau$ far away from the collision $\cC_+$. We will see that the forward and backward orbit of this point $z_1$ intersected with $\{g=0\}$, which is
given by 
\begin{equation}\label{def:TranslatedOrbit}
\|\ell_1+q\omega\|,\quad q=-\hat q^*\ldots \hat q^*\qquad \text{with }\hat q^*=q^*/10,
\end{equation}
does not hit the $4C\mu^\tau$-neighborhoods of $\cC_+$ and $\cC_-'$.

First we prove that the points in \eqref{def:TranslatedOrbit}  are away from the $4C\mu^\tau$ neighborhood of $\cC_+$. Indeed, since $\hat q^*\leq q^*$ we know that $\|q\omega\|\geq 20C\mu^\tau$ for all $q=-\hat q^*\ldots \hat q^*$ (see \eqref{eq:dioph-gamma}). Then, the distance from the collision $\mathcal C_+=(-4C\mu^\tau, 0)$ is 
\[
 \|\ell_1+q\omega+4C\mu^\tau\|\geq \|q\omega\|-\|\ell_1\|-4C\mu^\tau\geq 6C\mu^\tau.
\]

Now it only remains to prove that this orbit does not intersect the $4C\mu^\tau$-neighborhood of $\cC_-'$. 
We look first at the iterate which was too close to collision for $z_0$. That is, $q=q'$ which satisfied \eqref{def:Recu:CollisionHitting}. Then, for the orbit of $z_1$ we have
\[
 \|\ell_1+q'\omega-\ell''\|\geq 10C\mu^\tau-\|q'\omega-\ell''\|\geq 6C\mu^\tau
\]
Now we prove that for all other $q=-\hat q^*\ldots \hat q^*$ with $q\neq q'$ we are also far from collision. Indeed, there assume that there exists 
$q''=-\hat q^*\ldots \hat q^*$ with $q''\neq q'$ such that
\[
 \|\ell_1+q''\omega-\ell''\|\leq 4C\mu^\tau
\]
and we reach a contradiction. Indeed, 
\[
 \|(q'-q'')\omega\|\leq \|q'\omega-\ell''\|+\|\ell_1\|+\|\ell_1+q''\omega-\ell''\|\leq 18C\mu^\tau.
\]
Then, since $\omega\in B_\ga$ (see \eqref{eq:dioph-gamma},
\[
\frac{\ga}{2\hat q^*}\leq \frac{\ga}{|q'-q''|}\leq\|(q'-q'')\omega\|\leq 18C\mu^\tau
\]
This implies that 
\[
 \hat q^*\geq \frac{\ga\mu^{-\tau}}{36C}
\]
Nevertheless, by assumption 
\[
\hat q^*=\frac{q^*}{10}=\frac{1}{10}\left\lceil\frac{1}{20C}\ga\mu^{-\tau}-1\right\rceil.
\]
This completes the proof of Lemma \ref{lemma:DensityDelaunayrecurrent}. Note that  changing the number of forward and backward iterates from $q^*$ in \eqref{def:MaxIteratesRec} to $\hat q^*=q^*/10$ still leads to $\ga$-density of the forward and backward orbits.

\appendix

\section{The Delaunay coordinates}\label{app:Delaunay}

To have a self-contained paper, in this appendix we recall the definition of the Delaunay coordinates.
For $\mu=0$, system (\ref{r3b-e}) becomes (\ref{r3b-0})
\[
H_0(x,y)=\frac{|y|^2}{2}- x^t Jy-\frac{1}{|x|}.
\]
The Delaunay transformation is a symplectic transformation defined by
\[
\Psi(x,y)=(\ell,g,L,G)
\]
under which $H_0(x,y)$ becomes the totally integrable Hamiltonian
\[H_0(L,G)=-\frac{1}{2L^2}-G.
\]
One can construct the change of coordinates $\Psi$ in two steps.  First we take  the usual symplectic transformation  to polar coordinates
\begin{equation}\label{def:polars}
(x_1,x_2,y_1,y_2)=\Psi_1(r,\varphi,R,G)
\end{equation}
defined as 
\[
\left\{
\begin{split}
x_1&=r\cos\varphi \\
x_2&=r\sin\varphi\\
y_1&=R\cos\varphi-\frac{G}{r}\sin\varphi \\
y_2&=R\sin \varphi+\frac{G}{r}\cos \varphi
\end{split}
\right.
\]
The Hamiltonian in (\ref{r3b-0}) becomes
\[
H_0(r,R,\varphi,G)=\frac{R^2}{2}+\frac{G^2}{2r^2}-G-\frac{1}{r}
\]
Recall that $G$ is the angular momentum and itself is a first integral for the 2 body problem.
To obtain the Delaunay coordinates, to obtain Hamiltonian \eqref{0-r-a-a}, we consider a second symplectic transformation
\begin{equation}\label{trans}
(r,\varphi,R,G)=\Psi_2(\ell,g,L,G)
\end{equation}
where
\begin{itemize}
 \item $L=\sqrt a$ where $a$ is the semimajor axis of the ellipse.
 \item $G$ is the angular momentum.
 \item $\ell$ is the mean anomaly.
 \item $g$ is the argument of the perihelion with respect the primaries line.
\end{itemize}

The change of coordinates $\Psi_2$ is not fully explicit. Nevertheless, for some components it can be defined through successive changes of variables (for a more extensive explanation, one can see Appendix B.1 in \cite{FGKR}). For 
the position variables $(r,\varphi)$ one as 
%
\begin{equation}\label{def:rphiDelaunay}
\begin{split}
r=&r(\ell,L,G)=L^2(1-e\cos \mathfrak u(\ell))\\
\varphi&=\phi(\ell,g,L,G)=\mathfrak v(\ell)+g
\end{split}
\end{equation}
where $e=e(L,G)$ is the eccentricity defined in \eqref{def:eccentricity} the two functions $\mathfrak u(\ell)$ 
and $\mathfrak v(\ell)$ are implicitly defined by
\[
 \begin{split}
\ell&=\mathfrak u-e\sin \mathfrak u\\
\tan\frac{\mathfrak v}{2}&=\sqrt{\frac{1+e}{1-e}}\tan\frac{\mathfrak u}{2}.
\end{split}
\]
\vspace{10pt}

\section{Density estimate of the Diophantine numbers of constant type}\label{app:Diophantine}

Consider the set  of all Diophantine numbers with constant type satisfying
\eqref{eq:dioph-gamma}, which we have denoted by $B_\gamma$. We devote this 
appendix to 
prove the density of such set stated in Lemma 
\ref{lemma:densityconstantnumbers}.
%
Without loss of generality, we restrict on the $[0,1]$ interval and we prove 
that $B_\gamma$ is $O(\gamma)-$dense in it. We split the proof in several 
lemmas.
\begin{lem} For any $\gm>0$, there exists a constant $C(\gamma)$ satisfying 
\[\frac 
1\gamma-2\leq 
C(\gamma)\leq\frac 1\gamma\] such that, for any  $\om\in B_\gamma$,  the 
associated 
continuous fraction $\omega=[a_1,a_2,\cdots]$ satisfies
\[0\leq a_i\leq C(\gamma)\quad\text{ for  all }i\in\N.\]
\end{lem}
\begin{proof}To prove this lemma, consider the sequence of convergents of 
$\omega$, 
$\{\frac{p_n}{q_n}\}_{n\in\N}$, which is defined by
\[
\frac{p_n}{q_n}=[a_1,a_2,\cdots,a_n]
\]
The integers $p_n$, $q_n$ satisfy
\[
\begin{split}
p_n=a_np_{n-1}+p_{n-2},\quad n\geq2\\
q_n=a_nq_{n-1}+q_{n-2},\quad n\geq2
\end{split}
\]
where $p_0=a_0=0$, $p_1=1$, $q_0=1$ and $q_1=a_1$. They also satisfy 
\begin{equation}
\label{u-l-bound}
\frac{1}{q_n^2(2+a_{n+1})}<\frac{1}{q_n(q_n+q_{n+1})}\leq\Big|\om-\frac{p_n}{q_n}\Big|\leq\frac{1}{q_nq_{n+1}}<\frac{1}{q_n^2a_{n+1}}.
\end{equation}
For any $\om\in B_\gamma$, there exists 
$\gamma_\om\geq\gamma$, usually called  Diophantine constant, defined by
\[
\inf_{n\geq 0}|q_n(q_n\om-p_n)|=\gamma_\om.
\]
From \eqref{u-l-bound}, one has 
\[
\frac{1}{2+a_{n+1}}<
|q_n(q_n\om-p_n)|<\frac{1}{a_{n+1}}.
\]
Therefore, on the one hand
\[
\inf_{n\geq 1}\frac{1}{a_{n}}\geq\gamma_\om\geq \gamma.
\]
which implies $\sup_{n\geq 1}a_n \leq  \gamma^{-1}$. On the other hand, 
\[
 \inf_{n\geq 1}\frac{1}{a_{n}+2}\leq \gamma_\omega.
\]
which is equivalent to $\sup_{n\geq 1}a_{n}\geq \gamma_\omega^{-1}-2$. Taking 
the supremmum over all 
 $\om\in B_\gamma$ we obtain
 \[
 \sup_{\omega\in B_\gamma}\sup_{n\geq 1}a_{n}\geq \frac{1}{\gamma}-2
 \]
Therefore, we can conclude that
\[
\frac 1\gamma-2\leq C(\gamma)\leq\frac 1\gamma.
\] 
\end{proof}
The set $B_\gamma$ is a closed Cantor set (proved in \cite{P}). Therefore, it 
can be expressed as  $[0,1]\backslash 
B_\gamma=\bigcup_{i=1}^\infty(\alpha_i,\beta_i)$. 
We call $(\alpha_i,\beta_i)$ a 
gap of $B_\gamma$. The collection of the boundary points 
$\{\alpha_i,\beta_i\}_{i=1}^\infty$ is a countable set, which is ordered.

\begin{lem}\label{order}
Consider the set $\cC_K$  of all continuous fractions with entries upper 
bounded by a given  $K$. Then, formally we have 
$[0,1]\backslash\cC_K=\bigcup_{i=1}^\infty(\alpha_i,\beta_i)$ and each gap 
$(\alpha_i,\beta_i)$ can be  expressed either as
\begin{equation}\label{even}
(\alpha_i,\beta_i)=\Big([a_1,a_2,\cdots,a_m,L+1,K,1,K,1,\cdots],[a_1,a_2,\cdots,
a_m,L,1,K,1,K,\cdots]\Big)
\end{equation}
for some even  $m$ or
\begin{equation}\label{odd}
(\alpha_i,\beta_i)=\Big([a_1,a_2,\cdots,a_m,L,1,K,1,\cdots],[a_1,a_2,\cdots,a_m,
L+1,K,1,K,1,\cdots]\Big)
\end{equation}
for some odd  $m$. In both two cases 
$L\in\{1,2,\cdots,K-1\}$.
\end{lem}
\begin{proof}
Consider the  continuous fraction associated to a constant type number. Namely 
$\om=[a_1,a_2,\cdots]$ with each $a_i\in\{1,2\cdots, K\}$. Then, the one as the 
following  monotonicity: $\om$ decreases when increasing an odd entry and 
increases when decreasing an even entry. This gives a rule to order 
all the continuous fractions with $K-$bounded entries. Since $\cC_K$ does 
not intersect the  gaps 
$(\alpha,\beta)$, the first different entry of $\alpha$ and $\beta$ should 
have a difference by $1$. After that, it can be seen that the following entries 
must have consecutive values, as is shown in \eqref{even} and \eqref{odd}.
\end{proof}

\begin{cor}
The largest gap in 
$[0,1]\backslash\cC_K=\bigcup_{i=1}^\infty(\alpha_i,\beta_i)$ is
\[
\cG_{K}=\Big([2,K,1,K,\cdots],[1,1,K,1,K,\cdots]\Big).
\]
\end{cor}
\begin{proof}
In   Lemma \ref{order} we have shown that
\[
\begin{split}
0<\beta_i-\alpha_i=\,\text{diam}\,(\alpha_i,\beta_i)
<\,\Big|[a_1,a_2,\cdots,a_m,L]-[a_1,a_2,\cdots,a_m,L+1]\Big|\nonumber,
\end{split}
\]
where $a_i\geq 1$ for all $i=1,\cdots,m$. So the smaller $m$ is, the smaller 
the diameter of the gap. Thus the first different entry has to be $m=1$. Lemma 
\ref{order} gives all the other entries in the continuous fraction expansion.
\end{proof}

This corollary implies the proof of Lemma \ref{lemma:densityconstantnumbers}. 
Indeed $B_\gamma$ contains $\cC_K$ with 
\[
 K=\frac 1\gamma -2.
\]
Then, the  width of the largest gap in $B_\gamma$ cannot exceed the width of 
the interval
\[
\cG_{K}=\Big([2,K,1,K,\cdots],[1,1,K,1,K,\cdots]\Big),
\]
which is bounded by $O(1/K)$. Thus
$B_\gamma$ is at least $O(\gamma)$-dense in $[0,1]$.

\end{document}